\documentclass[12pt, A4]{article}
\usepackage{amssymb,amsfonts,mathrsfs,hyperref,microtype, amsthm}
\usepackage{amsmath}
\usepackage{bm, enumerate}
\usepackage{geometry}
\usepackage{eufrak}
\usepackage{graphicx}
\usepackage[T1]{fontenc}
\usepackage{rotating}
\usepackage{booktabs}
\usepackage{mathtools}
\usepackage{float}
\usepackage{breqn}
\usepackage{caption}
\usepackage{color}

\newtheorem{thm}{Theorem}[section]
\newtheorem{rem}{Remark}[section]
\newtheorem{definition}{Definition}[section]
\newtheorem{lem}{Lemma}[section]
\newtheorem{prop}{Proposition}[section]
\newtheorem{cor}{Corollary}[section]

\newtheorem{ex}{Example}[section]

\numberwithin{equation}{section}

\def\HH{ \EuFrak H}
\def\N{{\rm I\kern-0.16em N}}
\def\R{{\rm I\kern-0.16em R}}
\def \E{{\rm I\kern-0.16em E}}
\def\P{{\rm I\kern-0.16em P}}
\def\F{{\rm I\kern-0.16em F}}
\def\B{{\rm I\kern-0.16em B}}
\def\C{{\rm I\kern-0.46em C}}
\def\G{{\rm I\kern-0.50em G}}

\newcommand{\ud}{\mathrm{d}}

\numberwithin{equation}{section}

\font\eka=cmex10
\usepackage{ae}

\def\ind{\mathrel{\hbox{\rlap{%
\hbox to 7.5pt{\hrulefill}}\raise6.6pt\hbox{\eka\char'167}}}}
\begin{document}

\title{\Large{\bf A bound on the 2-Wasserstein distance between linear
  combinations of independent random variables}}

\date{}
\renewcommand{\thefootnote}{\fnsymbol{footnote}}

\author{Benjamin Arras, Ehsan Azmoodeh, Guillaume Poly and Yvik Swan}

\maketitle

\abstract 

We provide a bound on a natural distance between finitely and
infinitely supported elements of the unit sphere of
$\ell^2(\N^{\star})$, the space of real valued sequences with finite
$\ell^2$ norm. We use this bound to estimate the 2-Wasserstein
distance between random variables which can be represented as linear
combinations of independent random variables. Our results are
expressed in terms of a discrepancy measure which is related to
Nourdin and Peccati's Malliavin-Stein method. The main area of
application of our results is towards the computation of quantitative
rates of convergence towards elements of the second Wiener
chaos. After particularizing our bounds to this setting and comparing
them with the available literature on the subject (particularly the
Malliavin-Stein method for variance-gamma random variables), we
illustrate their versatility by tackling three examples: chi-squared
approximation for second order $U$-statistics, asymptotics for
sequences of quadratic forms and the behavior of the generalized
Rosenblatt process at extreme critical exponent.
% Something about the fact that second chaos is difficult, 2-Wasserstein
% distance does not allow for a tractable representation and Stein's
% method only works on a case-by-case basis. We propose a new technique
% based on ideas inspired by NP chaotic approach to SM and our
% techniques obviates necessity for solving Stein equations or [...] but
% comes at the expense of a strict structure on the things we compare.


\tableofcontents

\section{Introduction} 

% It is of obvious practical and theoretical importance to be able to
% precisely estimate the quality of the approximation between two
% probability distributions that are known to be close. The literature
% on large sample theory, for instance, abounds in quantitative
% estimates of the discrepancy between the distribution of relevant
% sample statistics ($M$-statistics, $U$-statistics, sample extremes,
% etc.), and suitably chosen target distributions (Gaussian, stable,
% extreme value, etc.).  The famous Berry-Ess\'een theorem, discovered
% circa 1940, is a quantitative version of the Central Limit theorem and
% provides estimates for the Kolmogorov distance between the law of a
% standardized sum of i.i.d.\ random variables and the normal
% distribution.  Another result in the same spirit is Nourdin and
% Peccati's fourth moment bound from \cite{n-pe-ptrf, n-pe-pams} which
% is a quantitative version of Nualart and Peccati's fourth moment
% theorem from \cite{nupe2005} and provides estimates for several
% probability distances (including Kolmogorov, total variation and
% Wasserstein-1) between the Gaussian distribution and the law of
% sequences of multiple Wiener-It\^o integrals.

In this paper, we provide bounds on the Wasserstein-2 distance (see
Definition~\ref{def:wassp}) ${{\bf \rm W}}_2(F_n,F_\infty)$ between
random variables $F_n$ and a target $F_{\infty}$ which satisfy the
following assumption.

\

\noindent \textbf{Assumption:}  There exist $q$ {non-zero} and pairwise
distinct real numbers $\{ \alpha_{\infty,k}\}_{1\le k \le q}$ as well
as sequences $\{\alpha_{n,k}\}_{n,k\ge 1} \subset \R$ such that
$\sum_{k=1}^q\alpha_{\infty,k}^2=\sum_k\alpha_{n,k}^2=1$ for all $n\ge1$ and 
\begin{equation}\label{eq:4}
F_n= \sum_{k\ge1} \alpha_{n,k} \, W_k  \mbox{ for all } n 
 \ge 1   \mbox{ and } F_{\infty}:= \sum_{k=1}^q \alpha_{\infty,k} W_k
\end{equation}
where  
the  $\{W_k\}_{k\ge1}$ is a sequence of i.i.d.\ random variables with mean 0, variance 1, {finite moments} of orders $2q+2$ and non-zero
$r$th cumulant for all $r=2,\cdots,2q+2$. 

\

In light of the coupling imposed by our Assumption it seems
intuitively evident that ${{\bf \rm W}}_2(F_n,F_\infty)$ ought to be
governed solely by the convergence rate of the approximating sequence
of coefficients $\{\alpha_{n,k}\}_{n,k\ge 1}$ towards
$\{ \alpha_{\infty,k}\}_{1\le k \le q}$. The main difficulty is to
identify the correct norm for this convergence and, following on
\cite{a-p-p}, we consider the quantity
\begin{equation}
  \label{eq:5}
  \Delta(F_n, F_{\infty}) = \sum_{k\ge1} \alpha^2_{n,k}
  \prod_{r=1}^{q} \left( \alpha_{n,k} - \alpha_{\infty,r} \right)^2.
\end{equation}
The main theoretical contribution of the paper is Theorem
\ref{thm:main-theorem1}, where we prove,
in essence, that under technical conditions on the limiting
coefficients we have the bound
\begin{equation}\label{eq:6}
  {{\bf \rm W}}_2(F_n,F_\infty) \le C \,\sqrt{\Delta(F_n, F_{\infty})},
\end{equation}
with $C>0$ is a constant depending only on $F_{\infty}$.

We comment briefly on the general strategy we adopt in order to obtain
a bound such as \eqref{eq:6}. Due to the structure imposed by our
Assumption on the random variables we consider, it is natural to bound
the $2$-Wasserstein metric by a quantity based on re-indexing
couplings. This leads us to considering a taylor-made norm $d_\sigma$
(see \eqref{defidist}) in a purely Hilbertian context. Then, based on
the careful analysis of minimization problems associated with
$d_\sigma$, we are able to identify bounding quantities which depend
polynomially on the coordinates of the sequences we want to compare
(see Theorems \ref{rationally-independent-hilbert} and
\ref{theo-dependant}). Recasting these quantities in the probabilistic
context we are interested in, we are able to link them to the
cumulants of the random variables $F_n$ and $F_\infty$ and finally to
obtain our main result.

The most important application of a bound such as \eqref{eq:6}
is that it provides quantitative rates of convergence towards elements
of the second Wiener chaos. Indeed it is a classical result that all
such random variables can be written as a linear combination of
centered chi-squared random variables, i.e.\ satisfy \eqref{eq:4} for
$W_k = Z_k^2-1$ and $\{ Z_k\}_{k \ge 1}$ i.i.d.\ standard normal
random variables. In Section~\ref{s:applications} we particularize our
general bounds to this setting and obtain the first rates of
convergence in Wasserstein-2 distance of sequences of elements
belonging to the second Wiener chaos, hereby complementing recent
contributions \cite{n-po-1,a-p-p} (see also \cite{Kr17} whose results
are posterior to a first version of this paper). Moreover, in Section \ref{sec--LowBound},
we obtain a general lower bound on the Wasserstein-2 distance between elements in the second 
Wiener chaos using the quantity $\Delta(F_n, F_{\infty})$. The rate exponent for this lower 
bound is $1$ leaving open the question of optimality of our bounds. We provide as well example where this 
lower bound can be refined (tightening the gap towards optimality). More importantly, these results emphasize the fact that the quantity $\Delta(F_n, F_{\infty})$ is the right one to study quantitative convergence results in 2-Wasserstein distance on the second Wiener chaos. Since the
intersection between the second chaos and the class of variance-gamma
distributed random variables is not empty it is also relevant to
detail our bounds in these cases. We perform this in Section
\ref{subsec:MS-VG}; this permits also direct comparison with
\cite{thale} where a similar setting was tackled - by entirely
different means.  

Finally, in Section \ref{sec:applications}, we apply
our bounds to three illustrative and relevant examples. First we consider chi-squared approximation
for second order U-statistics. We obtain among other results the
bound
\begin{align*}
{{\bf \rm W}}_2\big(nU_n(h), a(Z_1^2-1)\big)= \mathcal{O}(\dfrac{1}{\sqrt{n}}),
\end{align*}
for $U_n$ a second order U-statistics which has a degeneracy of order
1 (see Section \ref{rate1}).  

Next we consider the problem of
obtaining quantitative asymptotic results for sequences of quadratic
forms. Letting $\tilde{Q}_n(Z)=\sum_{i,j=1}^n\tilde{a}_{i,j}(n)Z_iZ_j$ and
$\tilde{Q}_\infty=\sum_{m=1}^q\tilde{\lambda}_m(Z_m^2-1)$ we deduce a general bound for ${\bf \rm W}_2\big(\tilde{Q}_n(Z)-\mathbb{E}[\tilde{Q}_n(Z)],
\tilde{Q}_\infty\big)$. In particular, for specific instances of the $n\times n$ 
 real-valued symmetric matrix $\big(\tilde{a}_{i,j}(n)\big)$, we obtain explicit rates of convergence:
\begin{align*}
{\bf \rm W}_2\big(\tilde{Q}_n(Z)-\mathbb{E}[\tilde{Q}_n(Z)],
  \tilde{Q}_\infty\big)=\mathcal{O}\big(\frac{1}{n^{\frac{\alpha}{2}}}\big)
\end{align*}
where $\alpha\in (0,1]$ (see Section \ref{Rate2}, Corollary \ref{Rate2}). Moreover, combining
Corollary \ref{Rate2} and an approximation rate in Kolmogorov distance
(Corollary \ref{Approx}) we are able to derive a quantitative
universality type result for quadratic forms defined by:
\begin{align*}
\tilde{Q}_n(X)=\sum_{i,j=1}^n\tilde{a}_{i,j}(n)X_iX_j
\end{align*}
with $(X_i)$ a sequence of i.i.d.\ random variables centered with unit variance and finite fourth moment (see Theorem \ref{UNI}). 

Finally, inspired by \cite{b-t}, we
consider the generalized Rosenblatt process at extreme critical
exponent. Letting
\begin{align*}
  Z_{\gamma_1,\gamma_2}=\int_{\mathbb{R}^2}\bigg(\int_0^1(s-x_1)^{\gamma_1}_+(s-x_2)^{\gamma_2}_+ds\bigg)dB_{x_1}dB_{x_2},
 \end{align*}
 with $\gamma_i\in(-1,-1/2)$ and $\gamma_1+\gamma_2>-3/2$ and
\begin{align*}
Y_{\rho}=\dfrac{a_{\rho}}{\sqrt{2}}(Z_1^2-1)+\dfrac{b_{\rho}}{\sqrt{2}}(Z_2^2-1),
  \quad 0 < \rho < 1
\end{align*}
we prove that 
\begin{equation*}
{{\bf \rm W}_2} (Z_{\gamma_1,\gamma_2},Y_{\rho}) \le C_\rho \, \sqrt{- \gamma_1 - \frac{1}{2}},
\end{equation*}
(see Lemma \ref{lem:cumasym}).

%[I do not mention the fact that only sequences in the second wiener
%chaos can converge to elements of the second wiener chaos. I'm not
%sure if this is true nor how to word it.]

In order to understand the significance of our  general bounds  and also to
contextualize the crucial quantity $\Delta(F_n, F_{\infty})$, it is
necessary at this stage to make a short digression into
Malliavin-Stein (a.k.a.\ Nourdin-Peccati) analysis.  Let $F_{\infty}$
be standard Gaussian and consider a sequence of normalized random
variables $F_n$ with sufficiently regular density with respect to the
Lebesgue measure. The \emph{Stein kernel} of $F_n$ is the random
variable $\tau_n(F_n)$ uniquely defined through the probabilistic
integration by parts formula
\begin{equation}
  \label{eq:8}
  E[\tau_n(F_n) \phi'(F_n) ] = E [F_n \phi(F_n)]
\end{equation}
 which is
supposed to hold for all smooth test functions $\phi:\R\to\R$. % The
% importance of the Stein kernel $\tau_F(F)$ has long been recognized,
% as attested for example by the research spawned from \cite[Lesson
% VI]{S86} or \cite{BU84, Kl85, CP89}.  Regularity assumptions for $F$
% to admit such a kernel are satisfied when $F$ is a multiple
% Wiener-It\^o integral, see e.g.\ \cite{n-pe-1}. 
The classical Stein identity, according to which
$E[\phi'(F_\infty) ] = E [F_\infty \phi(F_\infty)]$ for all smooth
$\phi$, implies in particular that the standard Gaussian distribution
as a Stein kernel which is constant and equal to 1. Hence
\begin{equation}
  \label{eq:2}
  \mathcal{S}(F_n, F_{\infty}) := E \left[ (\tau_n(F_n) - 1)^2 \right] = E
  \left[ \tau_n(F_n)^2 \right]-1
\end{equation}
necessarily captures some aspect of non-Gaussianity of $F_n$. As it
turns out this quantity -- called the Stein (kernel) discrepancy --
plays a crucial role in Gaussian analysis.  In particular, it has long
been known that $ \mathcal{S}(F_n, F_{\infty}) $ measures
non-Gaussianity quite precisely. First, see e.g.\ \cite[Lesson
VI]{S86} or \cite{BU84, Kl85, CP89}, it is equal to zero if and only
if $\mathcal{L}(F_n) = \mathcal{L}(F_{\infty})$ (equality in
distribution). Second, Stein's method also implies that
$ \mathcal{S}(F_n, F_{\infty})$ metrizes convergence in distribution,
i.e.\
\begin{equation}
  \label{eq:3}
  d_{\mathcal{H}}(F_n, F_{\infty}) = \sup_{h \in \mathcal{H}} E \left|
    h(F_n) - h(F_{\infty}) \right| \le \kappa_{\mathcal{H}} \sqrt{\mathcal{S}(F_n, F_{\infty})}
\end{equation}
for $ \mathcal{H}$ any class of sufficiently regular test functions
and $\kappa_{\mathcal{H}}$ a finite constant depending only on
$\mathcal{H}$; see \cite[Chapter 3]{n-pe-1} or \cite{LRS} for more
detail.  The breakthrough from \cite{n-pe-ptrf} is the discovery that
$\mathcal{S}(F_n, F_{\infty})$ is the linchpin of the entire theory of
``fourth moment theorems'' ensuing from the seminal paper
\cite{nupe2005}.  More precisely, Nourdin and Peccati were the first
to realize that the integration by parts formula for Malliavin
calculus could be used to prove 
\begin{equation}
  \label{eq:7}
  \mathcal{S}(F_n, F_{\infty}) \le {\frac{q-1}{3q} (E \left[ F_n^4 \right]-3)}
\end{equation}
whenever $F_n$ is an element of the $q$th Wiener chaos.  Combining
\eqref{eq:7} and \eqref{eq:3} thus provides quantitative fourth moment
theorems for chaotic random variables in integral probability metrics
including Total Variation, Kolmogorov and Wasserstein-1. We refer to
\cite{n-pe-ptrf} and the monograph \cite{n-pe-1} for a detailed
account; see also \cite{n-pe-pams} for an optimal-order bound (without
a square root), and \cite{Le12} for a general abstract version. % This
% striking result has, naturally, attracted a lot of attention and we
% invite the interested reader to consult the constantly updated webpage
% https://sites.google.com/site/malliavinstein/home for some insight in
% this very active and prolific field.
% \begin{rem}
% Nourdin and Peccati's initial Stein-type bound is in
% many cases not of the correct order and a different approach bypassing
% \eqref{eq:3} actually allowed them, in , to dispense
% with the square root and show optimality of the resulting bounds. A
% similar effect will be observed  in the context of this paper as well,
% as explained in Remark~\ref{rem:true-rate}.  
% \end{rem}
 
Stein kernels are not inherently Gaussian objects and are well
identified and tractable for a wide family of target distributions,
see e.g.\ \cite[Lesson VI]{S86}. It is therefore not unreasonable to
study, for $F_{\infty}$ having kernel $\tau_{\infty}(F_{\infty})$ and
satisfying general assumptions, the kernel discrepancy
$\mathcal{S}(F_n, F_{\infty}) := E \left[ (\tau_n(F_n) -
  \tau_{\infty}(F_{\infty}))^2 \right]$ in order to reap the
corresponding estimates from \eqref{eq:3}. This plan was already
carried out in \cite{n-pe-ptrf} for $F_{\infty}$ a centered gamma
random variable and pursued in \cite{viquez} and \cite{kusuoka} for
targets $F_{\infty}$ which were invariant distributions of diffusions.
Many useful target distributions do not, however, bear a tractable
Stein kernel and in this case the kernel discrepancy
$\mathcal{S}(F_n, F_{\infty})$ no longer captures relevant information
on the discrepancy between $\mathcal{L}(F_n)$ and
$\mathcal{L}(F_{\infty})$.  There is, for instance, an enlightening
discussion on this issue in \cite[pp 8-9]{thale} about the ``correct''
identity for the Laplace distribution which turns out to be
\begin{equation}\label{eq:9}
  E \left[ F_{\infty} \phi(F_{\infty}) \right] = E \left[  2
    \phi'(F_{\infty}) +  F_{\infty}
    \phi''(F_{\infty})\right]
\end{equation}
for smooth $\phi$. Identities involving second (or higher) order
derivatives of the test functions lead to considering higher order
versions of the Stein kernel, namely $\Gamma_1(F_n)$ defined through
$E \left[ F_n \phi(F_n) \right] = E \left[ \phi'(F_n) \Gamma_1(F_n)
\right]$ and $\Gamma_2(F_n)$ defined through
$E \left[ F_n \phi(F_n) \right] = E \left[ \phi'(F_n) \right]E \left[
  \Gamma_1(F_n) \right] + E \left[ \phi''(F_n) \Gamma_2(F_n) \right]$
where both identities are expected to hold for all smooth test
functions (higher order gamma's are defined iteratively). Applying the
intuition from Nourdin-Peccati analysis for Gaussian convergence  then
leads to a version of \eqref{eq:3} of the form
\begin{equation}
  \label{eq:10}
  d_{\mathcal{H}}(F_n, F_{\infty}) = \sup_{h \in \mathcal{H}} E \left|
    h(F_n) - h(F_{\infty}) \right| \le \kappa_{1,\mathcal{H}}
  \mathcal{S}_1(F_n, F_{\infty}) + \kappa_{2, \mathcal{H}}
  \mathcal{S}_2(F_n, F_{\infty})  
\end{equation}
where the constants $\kappa_{i,\mathcal{H}}, i=1, 2$ depend only on
$\mathcal{H}$ and $\mathcal{S}_i(F_n, F_{\infty}), i=1, 2$ provide a
comparison of the $\Gamma_i$ with the coefficients of the derivatives
appearing in the second order identities (e.g.\ \eqref{eq:9} in the
case of a Laplace target). Good bounds on the constants
$\kappa_{i,\mathcal{H}}, i=1, 2$ are crucial for \eqref{eq:10} to be
of use; such bounds require being able to solve specific (second
order) differential equations (called Stein equations) and providing
uniform bounds on these solutions and their derivatives.  This is
exactly the plan carried out in \cite{thale} for variance-gamma
distributed random variables, and their approach rests on the
preliminary work of \cite{g-thesis} who provides unified bounds on the
solutions to the variance-gamma Stein equations.

Aside from the variance-gamma case discussed in
\cite{g-variance-gamma,g-thesis}, there are several other recent
references where versions of \eqref{eq:8} and \eqref{eq:9} are
proposed for complicated probability distributions such as the
Kummer-$U$ distribution \cite{p-r-r}, or the distribution of products
of independent random variables \cite{g-2normal,g-2normal2, GMS}. The
common trait of all these is that the resulting identities all involve
second or higher order derivatives of the test functions. In
\cite{a-p-p-2} -- which is essentially based on the first part of a
previous version of this work -- we use Fourier analysis to obtain
identities for random variables of the form \eqref{eq:4} when
$\left\{ W_k \right\}_k$ is a sequence of gamma distributed random
variables. The resulting identities involve as many derivatives of the
test functions as there are different coefficients in the
decomposition \eqref{eq:4}. Applying the intuition outlined in the
previous paragraph leads to the realization that the quantity that
shall play the role of a Stein discrepancy
$\mathcal{S}(F_n, F_{\infty})$ in the context of random variables of
the form \eqref{eq:4} is exactly $\Delta(F_n, F_{\infty})$ defined in
\eqref{eq:5}.  We are therefore, in principle, in a position to use a
bound such as \eqref{eq:3} or \eqref{eq:10} to obtain rates of
convergence in integral probability metrics $d_{\mathcal{H}}$.  The
problem with this roadmap for as general a family as that described by
our Assumption is that the corresponding constants
$\kappa_{\mathcal{H}}$ are elusive save on a case-by-case basis for
specific choices of $F_{\infty}$.  This means in particular that
Nourdin and Peccati's version of Stein's method shall not provide
relevant bounds, at least at the present state of our knowledge on the
constants $\kappa_{\mathcal{H}}$, in one sweep for such a large family
as that concerned by our assumption \eqref{eq:4}.

In this paper we propose to only keep the relevant quantity
$\Delta(F_n, F_{\infty})$ whose importance to the problem was
identified thanks to the Nourdin-Peccati intuition, but then bypass
the difficulties inherited from the Stein methodology entirely. To
this end we choose to study the problem of providing bounds in terms
of an important and natural distance which is moreover better adapted
to our Assumption: the Wasserstein-2 distance which
we now define. 

\begin{definition}\label{def:wassp}
  Fix $p\geq 1$. % Given two probability measures $\nu$ and $\mu$ on the
%   Borel sets of $\R^d$ whose marginals have finite absolute moments of
%   order $p$, define the {\it Wasserstein distance} (of order $p$)
%   between $\nu$ and $\mu$ as the quantity
% $$
% {{\bf \rm W}}_p(\nu,\mu) \, = \, \inf_\pi 
% \bigg ( \int_{\R^d\times\R^d} |x-y|^p d\pi(x,y)\bigg ) ^{1/p}
% $$
% where the infimum runs over all probability measures $\pi$ on
% $\R^d\times \R^d$ with marginals $\nu$ and $\mu$.   
The Wasserstein
metric is  defined by
$${{\bf \rm W}}_p(F_n,F_{\infty}) \, = \, \big( \inf \E \Vert X-Y\Vert_{d}^p
\big)^{1/p}$$ where the infimum is taken over all joint distributions
of the random variables $X$ and $Y$ with respective marginals $F_n$
and $F_{\infty}$, and $\Vert \, \Vert_d$ stands for the Euclidean norm
on $\R^d$.
\end{definition}

  Relevant information about Wasserstein distances can be found,
  e.g.\ in \cite{villani-book}. We conclude this introduction by noting
  that, as  is well-known,  convergence
  with respect to ${\bf \rm W}_p$ is equivalent to the usual weak
  convergence of measures plus convergence of the first $p$th
  moments. Also, a direct application of H\"older inequality implies
  that if $1\le p \le q$ then ${\bf \rm W}_p \le {\bf \rm
    W}_q$. Finally, we mention that the 2-Wasserstein distance is not of the
  family of integral probability metrics $d_{\mathcal{H}}$ (recall
  \eqref{eq:3} for a definition). 
% {\color{blue} It is also well known that the 2-Wasserstein distance
%   admits the following dual representation  
% $$ d_{W_2} (\nu,\mu)^2 = \sup_{\varphi \in C^{\text{Lip}}_b (\R^d)}
% \Big( \int_{\R^d} \varphi^\star \ud \nu - \int_{\R^d} \varphi \ud \mu
% \Big) $$ where $\varphi^\star (x )= \inf_{y \in \R^d} \varphi(y) +
% \vert x-y \vert^2$.} 
% }

%%%%%%%%%%%%%%%%%%%%%%%%%%%%   Section 2   %%%%%%%%%%%%%%%%%%%%%%%%%%%%%%%%%%%

\section{Wasserstein-2 distance between linear combinations}\label{s:main-result}

\subsection{A general result on Hilbert spaces}\label{sec:Hilbert-result}
We denote by $\ell^2(\N^\star)$ the space of real valued sequences
$u=(u_n)_{n\geq 1}$ such that $\sum_{n=1}^{\infty}u_n^2<\infty$. It is
a Hilbert space endowed with the natural inner product and induced
Euclidean norm $\|\cdot\|_2$.  We aim to measure distances between
elements $x, y$ of the unit sphere of $\ell^2(\N^\star)$ where $x$ is
a finitely supported sequence $x=(x_1,\cdots,x_q,0,0,\cdots)$ and
$y=(y_i)_{i\ge 1}$ is arbitrary. Denoting $\sigma(\N^\star)$ the set
of permutations of $\N^\star$, we introduce the following distance
between $x$ and $y$:
\begin{equation}\label{defidist}
d_\sigma(x,y)=\min_{\pi \in \sigma(\N^\star)} \|x-y_\pi\|_2=\min_{\pi
  \in \sigma(\N^\star)}\left(\sum_{i=1}^\infty
  (x_i-y_{\pi(i)})^2\right)^{\frac 1 2}. 
\end{equation}
Now we define the polynomial $Q_x ( t )=t^2 \prod_{i=1}^q
(t-x_i)^2$. Then, we have the following Theorem.
\begin{thm}\label{rationally-independent-hilbert}
  Suppose that $\left(x_1^2,\cdots,x_q^2\right)$ are {rationally
    independent}. Then there exists a constant $C_x$ which only
  depends on $x$ such that for any $y$ in the unit sphere of
  $\ell^2(\N^\star)$ we get
\begin{equation}\label{bound-fundamental}
d_\sigma(x,y)\le C_x\sqrt{\sum_{i=1}^\infty Q_x (y_i)}.
\end{equation}
\end{thm}
\begin{proof}
We first notice that $$ \min_{t\in\R}\left(\frac{1}{t^2} Q_x(t)+\sum_{i=1}^q \frac{Q_x(t)}{(t-x_i)^2}\right):=\delta_x>0.$$
As a result, for any real number $t$, at least one of the following inequalities is true.
\begin{eqnarray*}
\textit{ineq}_0:&\,\,t^2&\le \frac{q+1}{\delta_x} Q_x ( t ),\\
\textit{ineq}_1:&\,\,(t-x_1^2)&\le \frac{q+1}{\delta_x} Q_x(t),\\
 &&\vdots\\
\textit{ineq}_q:&\,\,(t-x_q)^2&\le\frac{q+1}{\delta_x} Q_x(t).
\end{eqnarray*}
Although several of the aforementionned inequalities can hold
simultaneously, one may always associate to any integer $i\ge 1$ some
index $l$ in $\{0,1,\cdots,q\}$ such that $\textit{ineq}_l$ holds for
$t=y_i$. Hence, one may build a partition of
$\N^\star=I_0\cup I_1\cup\cdots\cup I_q$ such that
$$\left\{
\begin{array}{l}
\forall i \in I_0,\,\,y_i^2 \le \frac{q+1}{\delta_x} Q_x(y_i)\\
\forall j \in \{1,\cdots,q\},\,\forall i \in I_j,\,\,(y_i-x_j)^2 \le \frac{q+1}{\delta_x} Q_x(y_i).
\end{array}
\right.
$$
Note that for $j\in \{1,...,q\}$ we have $\# I_j<\infty$. Indeed, if one assumes, for example, that $\# I_1=+\infty$, then one necessarily has that $x_1=0$ (which is a contradiction).
This entails the following bound
\begin{equation}\label{bound-partition}
\sum_{i \in I_0} y_i^2 +\sum_{j=1}^q \sum_{i \in I_j} (y_i-x_j)^2 \le \frac{q+1}{\delta_x} \sum_{i=1}^\infty Q_x(y_i).
\end{equation}
For any integer $i\ge 1$, we set $z_i=x_j$ if $i \in I_j$ for $j \in \{1,\cdots,q\}$ and we set $z_i=0$ when $i\in I_0$. Using triangle inequality and (\ref{bound-partition}) we may infer that
\begin{equation}\label{bound-norme}
\left|\|z\|_2-1\right|=\left| \sqrt{ \sum_{j=1}^q \# I_j x_j^2}-\|y\|_2 \right| \le \sqrt{\sum_{i \in I_0} y_i^2 +\sum_{j=1}^q \sum_{i \in I_j} (y_i-x_j)^2}\le \sqrt{\frac{q+1}{\delta_x} \sum_{i=1}^\infty Q_x(y_i)}.
\end{equation}
We need to introduce the following quantity
\begin{equation}\label{seuil}
\eta:=\min \left\{\left|\sqrt{\sum_{j=1}^q n_j x_j^2}-1\right| \,\,;\,\,(n_1,\cdots,n_q)\in (\N^\star)^q/\{(1,1,\cdots,1)\}\right\}.
\end{equation}
Since we do not let $(n_1,\cdots,n_q)=(1,\cdots,1)$ in the above minimization, and owing to the assumption of rational independence of $(x_1^2,\cdots,x_q^2)$, it follows that $\eta>0$. Relying on the bound (\ref{bound-norme}), one has the following implication
\begin{eqnarray*}
\sqrt{\frac{q+1}{\delta_x} \sum_{i=1}^\infty Q_x(y_i)}<\eta &\,\Rightarrow\,& \# I_1=\#I_2=\#I_3=\cdots=\#I_q=1\\
&\Rightarrow& \|x-y_\pi\|_2 \le \sqrt{\frac{q+1}{\delta_x} \sum_{i=1}^\infty Q_x(y_i)},
\end{eqnarray*}
for $\pi$ being any permutation of $\N^\star$ satisfying 
$$I_1=\{\pi(1)\},\cdots,I_q=\{\pi(q)\}, I_0=\pi\Big{(}\{q+1,q+2,\cdots\}\Big{)}.$$
Finally, it holds
\begin{equation}\label{boundquasifinal}
\sqrt{\frac{q+1}{\delta_x} \sum_{i=1}^\infty Q_x(y_i)}<\eta \,\Rightarrow\, d_\sigma(x,y) \le \sqrt{\frac{q+1}{\delta_x} \sum_{i=1}^\infty Q_x(y_i)},
\end{equation}
which implies that (given the trivial bound $d_\sigma(x,y)\le 2$)
\begin{equation}\label{endbound}
d_\sigma(x,y) \le (1+\frac{2}{\eta})\sqrt{\frac{q+1}{\delta_x} \sum_{i=1}^\infty Q_x(y_i)}.
\end{equation}
The proof is then achieved with the constant $C_x =(1+\frac{2}{\eta}) \sqrt{\frac{q+1}{\delta_x}}.$
\end{proof}
\noindent
Let us now deal with the case when $(x_1^2,\cdots,x_d^2)$ are not anymore rationally independent. In this situation, one might write $1=\sum_{j=1}^q n_j x_j^2$ for several choices of vectors $(n_1,\cdots,n_q) \in (\N^\star)^q$. We must introduce the set of all these choices, namely:
$$E=\left\{ \bold{n}:=(n_1,\cdots,n_q) \in (\N^\star)^q \left|\,\sum_{j=1}^q n_j x_j^2=1\right.\right\}.$$
Besides, for any $\bold{n}=(n_1,\cdots,n_q) \in E$ we define the following element of the unit sphere of $\ell^2(\N^\star)$:
$$x_{\bold{n}}=( \underbrace{x_1, \ldots, x_1}_{ n_1 \text{times}}, \ldots, \underbrace{x_q, \ldots, x_q}_{ n_q\text{times}},0,0,\cdots).$$
We then have the following Theorem.
\begin{thm}\label{nbrefiniadherencevalues}
There exists a constant $C_x$ only depending on $x$ such that for any $y$ in the unit sphere of $\ell^2(\N^\star)$ we get:
$$
\min \left\{d_\sigma(x_{\bold{n}},y)\,;\,\bold{n}\in E\right\}\le C_x\sqrt{\sum_{i=1}^\infty Q_x (y_i)}.
$$
\end{thm}
\begin{proof} We proceed as in the proof of Theorem \ref{rationally-independent-hilbert}, from its begining until the bound (\ref{seuil}). The only difference is that we must now consider
\begin{equation}\label{seuilbis}
\kappa:=\min \left\{\left|\sqrt{\sum_{j=1}^q n_j x_j^2}-1\right| \,\,;\,\,(n_1,\cdots,n_q)\in (\N^\star)^q/E\right\}.
\end{equation}
Similarly, since we removed $E$ from the above minimization problem, it follows that $\kappa>0$. Relying on the bound (\ref{bound-norme}), one has the following implication
\begin{eqnarray*}
\sqrt{\frac{q+1}{\delta_x} \sum_{i=1}^\infty Q_x(y_i)}<\kappa &\,\Rightarrow\,& (\# I_1,\# I_2,\# I_3,\cdots, \# I_q)\in E\\
&\Rightarrow& \exists \bold{n} \in E,\,\|x_{\bold{n}}-y_\pi\|_2 \le \sqrt{\frac{q+1}{\delta_x} \sum_{i=1}^\infty Q_x(y_i)},
\end{eqnarray*}
for $\pi$ being any permutation of $\N^\star$ satisfying 
$$
\left\{
\begin{array}{lcl}
I_1&=&\pi(\{1,\cdots,n_1\}),\\
I_2&=&\pi(\{n_1+1,\cdots,n_1+n_2\}),\\
\vdots&=&\vdots\\
I_q&=&\pi(\{n_1+\cdots+n_{q-1}+1,\cdots,n_1+\cdots+n_q\}),\\
I_0&=&\pi\Big{(}\{n_1+\cdots+n_q+1,n_1+\cdots+n_q+2,\cdots\}\Big{)}.
\end{array}
\right.
$$
Finally, it holds
\begin{equation}\label{boundquasifinal}
\sqrt{\frac{q+1}{\delta_x} \sum_{i=1}^\infty Q_x(y_i)}<\kappa \,\Rightarrow\, \min \left\{d_\sigma(x_{\bold{n}},y)\,;\,\bold{n}\in E\right\} \le \sqrt{\frac{q+1}{\delta_x} \sum_{i=1}^\infty Q_x(y_i)},
\end{equation}
which can also be written
\begin{equation}\label{endbound}
\min \left\{d_\sigma(x_{\bold{n}},y)\,;\,\bold{n}\in E\right\}\le (1+\frac{2}{\kappa})\sqrt{\frac{q+1}{\delta_x} \sum_{i=1}^\infty Q_x(y_i)}.
\end{equation}
The proof is then achieved with the constant $C_x =(1+\frac{2}{\kappa}) \sqrt{\frac{q+1}{\delta_x}}.$
\end{proof}
\noindent
In the above situation, the quantity $\sum_{i=1}^\infty Q_x(y_i)$ is not
sufficient anymore to ensure the uniqueness of the limit for the
convergence for the metric $d_\sigma(\cdot,\cdot)$. There may be
several adherence values and some additional information is then
required. Set
$$\displaystyle{\Delta_{p,x}(y)=|\sum_{i=1}^{\infty}
  (y_i^p-x_i^p)|}.$$ 
We have the following Theorem.
\begin{thm}\label{theo-dependant}
There exists a constant $\tilde{C}_x$ which only depends on $x$ such that, for any $y$ with $\|y\|_2=1$, we get
$$d_\sigma(x,y)\le \tilde{C}_x \left(\sqrt{\sum_{i=1}^\infty Q_x(y_i)}+\max_{3\le s \le q+1} \Delta_{s,x}(y)\right).$$
\end{thm}
\begin{proof}
  Relying on Theorem \ref{nbrefiniadherencevalues}, it holds that
$$
\min \left\{d_\sigma(x_{\bold{n}},y)\,;\,\bold{n}\in E\right\}\le C_x\sqrt{\sum_{i=1}^\infty Q_x (y_i)}.
$$
Note that it is not assumed that the real numbers $(x_1,\cdots,x_q)$ are pairwise distinct. We can extract a subsequence $(u_1,\cdots,u_s)$ with $s\le q$ by removing the possible repetitions. For any $\bold{n}\in E$, let us also denote by $m_i(\bold{n})$ the number of repetitions of $u_i$ among the sequence $x_{\bold{n}}$ and by $m_i$ the number of repetitions in the sequence $x$. Thus, we have
$$\forall p\in \{3,\cdots,q+1\},\,\,\sum_{i=1}^\infty x_{\bold{n}}(i)^{p}=\sum_{i=1}^s m_i(\bold{n}) u_i^p.$$
Suppose that
$\bold{n}=\text{argmin}\left\{d_\sigma(x_{\bold{n}},y)\,;\,\bold{n}\in
  E\right\}$, by the triangle inequality get for all $p\in \{3,\cdots,q+1\}$,
\begin{eqnarray*}
\left|\sum_{i=1}^s
  (m_i(\bold{n})-m_i)u_i^p\right|
&\le& \left|\sum_{i=1}^s m_i(\bold{n}) u_i^p-\sum_{i=1}^\infty y_i^p\right|+\Delta_{p,x}(y)\\
&=& \left|\sum_{i=1}^\infty \left(x_{\bold{n}}(i)-y_i\right)\left(\sum_{j=0}^{p-1}x_{\bold{n}}(i)^j y_i^{p-1-j}\right) \right|+\Delta_{p,x}(y)\\
&\le& p \sum_{i=1}^\infty\left|x_{\bold{n}}(i)-y_i\right|\left(|x_{\bold{n}}(i)|+|y_i|\right)+\Delta_{p,x}(y)\\
&\stackrel{\le}{\tiny{\text{Cauchy-Schwarz}}}& 2 p~d_\sigma(x_{\bold{n}},y)+\Delta_{p,x}(y)\\
&\le & 2 p C_x\sqrt{\sum_{i=1}^\infty Q_x (y_i)}+\Delta_{p,x}(y).
\end{eqnarray*}
Finally, set $\bold{V}:=V(u_1,\cdots,u_s)=\text{mat}\left(u_i^j\right)_{1\le i\le s, 0\le j \le s-1}$ the Vandermonde matrix associated to the pairwise distinct real numbers $(u_1,\cdots,u_s)$ and $\vec{m}=\left((m_i(\bold{n})-m_i)u_{i}^2\right)_{1\le i \le s}$. The above inequality reads as 
$$\|{}^t \vec{m} \bold{V}\|_\infty \le 2 (q+1) C_x\sqrt{\sum_{i=1}^\infty Q_x (y_i)}+\sup_{3\le p \le q+1}\Delta_{p,x}(y)
.$$
Now, we set
$$\alpha_x=\min\left( \|{}^t \bold{k}\bold{V}\|_\infty~\big{|}~ \bold{k}=(k_1 u_1^2,\cdots,k_s u_s^2)~\text{and}~(k_1,\cdots,k_s) \in \mathbb{Z}^s/\{(0,\cdots,0)\}\right),$$ since $\bold{V}$ is invertible we must have $\alpha_x>0$. That is why,
$$2 (q+1) C_x\sqrt{\sum_{i=1}^\infty Q_x (y_i)}+\sup_{3\le p \le r+1}\Delta_{p,x}(y)<\alpha_x\Rightarrow \vec{m}=0.$$
In the latter situation we also get $m_i(\bold{n})=m_i$, $x_{\bold{n}}=x$ and of course the desired bound
$$2 (q+1) C_x\sqrt{\sum_{i=1}^\infty Q_x (y_i)}+\sup_{3\le p \le r+1}\Delta_{p,x}(y)<\alpha_x\Rightarrow d_{\sigma}(x,y)\le C_x\sqrt{\sum_{i=1}^\infty Q_x (y_i)}.$$
The proof is now achieved with $\tilde{C_x}=2(q+1)C_x (1+\frac{2}{\alpha_x})$.
\end{proof}

\subsection{A probabilistic interpretation}

Let us give $\{W_k\}_{k \ge 1}$ an i.i.d. sequence of random variables admitting moments of orders $r=2,\cdots,2q+2$ and which satisfies $\E(W_1)=0, \E(W_1^2)=1$. We shall further assume that {all cumulants of orders $r=2,\cdots,2q+2$ are not zero}. We set
\begin{equation}\label{main-prob-objects}
F_n=\sum_{k=1}^\infty \alpha_{n,k} W_k\,\,,\,\, F_\infty=\sum_{k=1}^q \alpha_{\infty,k}W_k,
\end{equation}
for $ (\alpha_{n,k})_{k\geq 1}$ and $(\alpha_{\infty,k})_{k\geq 1}$ two sequences of real numbers. We also assume that:
$$\sum_{k=1}^\infty\alpha_{n,k}^2=\sum_{k=1}^q\alpha_{\infty,k}^2=1.$$
Using standard properties of cumulants one has for any $r=2,\cdots,2q+2$:
\begin{equation*}
\kappa_r(F_n)=\kappa_r(W_1)\sum_{k=1}^\infty \alpha_{n,k}^r \,\,,\,\,\kappa_r(F_\infty)=\kappa_r(W_1)\sum_{k=1}^q \alpha_{\infty,k}^r.
\end{equation*}

\begin{lem}\label{lem:linear-combiniation}
For any $n \in \N$ we have
\begin{equation} \label{eq:Delta-general}
\begin{split}
\Delta(F_n,F_\infty)=\Delta(F_n):&= \sum_{k\ge1} \alpha^2_{n,k} \prod_{r=1}^{q} \left( \alpha_{n,k} - \alpha_{\infty,r} \right)^2,\\
&=\sum_{r=2}^{2q+2} \Theta_r \sum_{k\ge1} \alpha^r_{n,k},\\
&=\sum_{r=2}^{2q+2} \Theta_r \dfrac{\kappa_r(F_n)}{\kappa_r(W_1)}.
\end{split}
\end{equation}
where the coefficients $\Theta_r$ are the coefficients of the polynomial
\begin{equation}\label{eq:polynomial}
Q_{\alpha_{\infty}}(x)=(P(x))^2= (x \prod_{i=1}^q (x-\alpha_{\infty,i}))^2.
\end{equation}
\end{lem}
\noindent
From a probabilistic point of view, Theorems \ref{rationally-independent-hilbert} and \ref{theo-dependant} take the following form:
\begin{thm}\label{thm:main-theorem1}
If the real numbers $\{\alpha_{\infty,r}^2\}_{0\le r\le q}$ are rationnally independent then
\begin{equation}\label{eq:main-estimate}
{{\bf \rm W}}_2(F_n,F_\infty) \le C \,\sqrt{\Delta(F_n)}  \qquad \forall n\ge1,
\end{equation}
if they are not, one instead gets
\begin{equation}\label{eq:main-estimate-bis}
{{\bf \rm W}}_2 (F_n,F_\infty) \le C
\,\left(\sqrt{\Delta(F_n)}+\sum_{r=2}^{q+1}|\kappa_r(F_n)-\kappa_r(F_\infty)|\right)
\qquad \forall n\ge1 
\end{equation}
where the constant $C$ depends only, in both cases, of the target $F_\infty$.
\end{thm}

\begin{proof}
The proof is a direct consequence of Theorems \ref{rationally-independent-hilbert} and \ref{theo-dependant}. Indeed, set $\alpha_n=\{\alpha_{n,k}\}_{k\ge 1}$ and $\alpha_\infty=\{\alpha_{\infty,k}\}_{k\ge 1}$, by definition of the 2-Wasserstein distance, we get ${{\bf \rm W}}_2(F_n,F_\infty)\le d_\sigma\left(\alpha_n,\alpha_\infty\right)$. As before, we set $Q_{\alpha_\infty}(x)=x^2\prod_{k=1}^q(x-\alpha_{\infty,k})^2$. Finally, recalling that $\sum_{k=1}^\infty Q_{\alpha_\infty}(\alpha_{n,k})=\Delta(F_n)$, the result follows.
\end{proof}
\begin{rem}
  An important question concerning the sharpness of the estimate
  (\ref{eq:main-estimate}) was raised by referees on a previous
  version of this paper. We first notice that for some appropriate
  constant $C>0$ and for all $x\in[-1,1]$, one gets
  $Q_{\alpha_\infty}(x)\le C x^2$ and for all $k=1,\cdots,q$,
  $Q_{\alpha_\infty}(x)\le C (x-\alpha_{\infty,k})^2$. Hence, we may
  deduce that
$$\Delta(F_n)=\sum_{k=1}^\infty Q_{\alpha_\infty}(\alpha_{n,k})\le C d_\sigma(\alpha_n,\alpha_\infty)^2,$$
and the result follows since one gets, for appropriate constants $A,B>0$ that
$$A \, d_\sigma(\alpha_n,\alpha_\infty)\le \sqrt{\Delta(F_n)} \le B \, d_\sigma(\alpha_n,\alpha_\infty).$$
Unfortunately, at present, we are unable to say whether distance $d_\sigma$ is equivalent to the 2-Wasserstein distance. Nonetheless, in the context of second Wiener chaos, we provide a general lower bound on the 2-Wasserstein distance in Section \ref{sec--LowBound} as well as a simple example which refines this lower bound.
\end{rem}

%\begin{lem}\label{lem:appendix}
%For the vector $\bm{ \alpha_\infty} = (\alpha_{\infty,1}, \cdots,\alpha_{\infty,q}) \in \R^q$ where $\alpha_{\infty,i}$ are non-zero and distinct, we denote  
%\begin{equation*}
%d(x,\bm{ \alpha_\infty} ):= \min_{i=1,\cdots,q} \, \vert x - \alpha_{\infty,i} \vert, \qquad \forall \, x \in \R.
%\end{equation*}
%Then, there exists a constant $M$ such that
%\begin{equation*}
%d(x, \bm{ \alpha_\infty})^2 \le M \prod_{i=1}^q (x - \alpha_{\infty,i})^2.
%\end{equation*}
%\end{lem}
%\begin{proof}
%Consider the function $f:\R - \{ \alpha_{\infty,1},\cdots, \alpha_{\infty,q} \} \to \R$ given by
%\begin{equation*}
%f(x):= \frac{\prod_{i=1}^q (x - \alpha_{\infty,i})^2}{d(x,  \bm{ \alpha_\infty})^2}.
%\end{equation*}
%Then, obviously $f$ is a continuous function on
%$\R - \{ \alpha_{\infty,1},\cdots, \alpha_{\infty,q} \}$ and can be
%extended to a continuous function on whole real line $\R$ by setting
%$f(\alpha_{\infty,i}):= \prod_{j \neq i} (\alpha_{\infty,j} -
%\alpha_{\infty,i})^2 \neq 0$
%at each point $\alpha_{\infty,i}$ for $i=1,\cdots,q$. On the other
%hand, note that we have $f(x) \to \infty$ as $\vert x \vert$ tends to
%infinity. Hence, $f$ is bounded from below by a positive constant, say
%$M$.
%\end{proof}

\subsection{Specializing to the second Wiener chaos}\label{s:applications}

In this section, we apply our main results in a desirable framework
when the approximating sequence $F_n$ are elements of the second
Wiener chaos of the isonormal process $\rm X=\{X(h); \ h \in \HH\}$
over a separable Hilbert space $\HH$.  We refer the reader to
\cite{n-pe-1} Chapter 2 for a detailed discussion on this
topic. Recall that the elements in the second Wiener chaos are random
variables having the general form $F=I_2(f)$, with
$f \in \HH^{\odot 2}$. Notice that, if $f=h\otimes h$, where
$h \in \HH$ is such that $\Vert h \Vert_{\HH}=1$, then using the
multiplication formula one has $I_2(f)=\rm X (h)^2 -1 = N^2 -1$
(equality in distribution), where $N \sim \mathscr{N}(0,1)$. To any
kernel $f \in \HH^{\odot 2}$, we associate the following
\textit{Hilbert-Schmidt} operator
\begin{equation*}
A_f : \HH \mapsto \HH; \quad g \mapsto f\otimes_1 g. 
\end{equation*}
%It is also convenient to introduce the sequence of auxiliary kernels
%\begin{equation}
%\left\{ f\otimes _{1}^{\left( p\right) }f:p\geq 1\right\} \subset \mathfrak{H%
%}^{\odot 2}  \label{kern1}
%\end{equation}%
%defined as follows: $f\otimes _{1}^{\left( 1\right) }f=f$, and, for $p\geq 2$,%
%\begin{equation}
%f\otimes _{1}^{\left( p\right) }f=\left( f\otimes _{1}^{\left(
%p-1\right) }f\right) \otimes _{1}f\text{.}  \label{kern2}
%\end{equation}
%In particular,$f\otimes _{1}^{\left( 2\right) }f=f\otimes _1 f$. 
We also write $\{\alpha_{f,j}\}_{j \ge 1}$ and $\{e_{f,j}\}_{j \ge 1}$, respectively, to indicate the (not necessarily distinct) eigenvalues of $A_f$ and the corresponding eigenvectors. We remind that $F_\infty$ is defined by:
\begin{align}
F_\infty=\sum_{j=1}^q\alpha_{\infty,j}(N_j^2-1)
\end{align}
where $\{N_j,\ j\in \{1,...,q\}\}$ is a collection of i.i.d. standard normal random variables. The next proposition gathers some relevant properties of the elements of the second Wiener chaos associated to $\rm X$.

\begin{prop}[See Section 2.7.4 in \cite{n-pe-1} and Lemma 3.1 in \cite{a-p-p} ] \label{second-property}
Let $F=I_{2}(f)$, $f \in \HH^{ \odot 2}$, be a generic element of the second Wiener chaos of $\rm X$, and write $\{\alpha_{f,k}\}_{k\geq 1}$ for the set of the eigenvalues of the 
associated Hilbert-Schmidt operator $A_f$.

\begin{enumerate}
 \item The following equality holds: $F=\sum_{k\ge 1} \alpha_{f,k} \big( N^2_k -1 \big)$, where $\{N_k\}_{k \ge 1}$ is a sequence of i.i.d. $\mathscr{N}(0,1)$ random variables that are elements of the isonormal process $\rm X$, and the series converges in $L^2$ and almost surely.
 \item For any $r\ge 2$,
 \begin{equation*}
  \kappa_r(F)= 2^{r-1}(r-1)! \sum_{k \ge 1} \alpha_{f,k}^r.%= 2^{i-1}(i-1)! \times \langle f \otimes^{(i-1)}_{1}f ,f \rangle_{\HH^{\otimes 2}}.
 \end{equation*}
\item For polynomial $Q_{\alpha_{\infty}}$ as in $(\ref{eq:polynomial})$ we have $\Delta(F) = \sum_{k\geq 1} Q_{\alpha_{\infty}}(\alpha_{f,k})$. In particular $\Delta(F_\infty)=0$.
\end{enumerate}
\end{prop}

The next corollary is a direct application of our main finding, namely Theorem \ref{thm:main-theorem1}, and provides quantitative bounds for the main results in \cite{n-po-1,a-p-p}.
\begin{cor}\label{cor:2wiener}
Assume that the normalized sequence $F_n=\sum_{k\ge 1} \alpha_{n,k} \big( N^2_k -1 \big)$ belongs to the second Wiener chaos associated to the isonormal process $\rm X$, and the target random variable $F_\infty$ as in (\ref{main-prob-objects}) with $W_k = N^2_k -1$ where $\{N_k\}_{k \ge 1}$ is a sequence of i.i.d. $\mathscr{N}(0,1)$ random variables. Then there exists a constant $C>0$ depending only on the target random variable $F_\infty$ (and hence independent of $n$) such that 
\begin{enumerate}
\item[(a)] 
$${{\bf \rm W}_2}(F_n,F_\infty) \le \, C \, \bigg( \sqrt{\Delta(F_n)} + \sum_{r=2}^{q+1} \vert \kappa_r(F_n) - \kappa_r(F_\infty) \vert \bigg).$$
\item[(b)] if moreover $\dim_{\mathbb{Q}} \text{span} \{
  \alpha^2_{\infty,1},\cdots,\alpha^2_{\infty,q} \} =q$, then ${{\bf
      \rm W}_2}(F_n,F_\infty) \le \, C \, \sqrt{\Delta(F_n)}$. This
  implies that the sole convergence $\Delta(F_n) \to
  \Delta(F_\infty)=0$ is sufficient for convergence in distribution
  towards the target random variable $F_\infty$. 
\end{enumerate}
\end{cor}
\begin{rem}\label{rem:thale}{ \rm The upper bound in Corollary
    \ref{cor:2wiener}, part (a) requires the separate convergences of
    the first $q+1$ cumulants for the convergence in distribution
    towards the target random variable $F_\infty$ as soon as
    $\dim_{\mathbb{Q}} \text{span} \{
    \alpha^2_{\infty,1},\cdots,\alpha^2_{\infty,q} \} < q$. This is
    consistent with a quantitative result in \cite{thale}, see also
    Section \ref{subsec:MS-VG} below. In fact, when $q=2$ and
    $\alpha_{\infty,1}=- \alpha_{\infty,2}=1/2$, then the target
    random variable $F_\infty$ $( = N_1 \times N_2$, where
    $N_1,N_2 \sim \mathscr{N}(0,1)$ are independent and equality holds
    in law) belongs to the class of {\it Variance--Gamma}
    distributions $VG_c(r,\theta,\sigma)$ with parameters $r=\sigma=1$
    and $\theta=0$. Then, \cite[Corollary 5.10, part (a)]{thale} reads

\begin{equation}\label{eq:thale-bound}
{{\bf \rm W}_1} (F_n,F_\infty) \le C\, \sqrt{\Delta(F_n) + 1/4 \, \kappa^2_3(F_n)}.
\end{equation}
Therefore, for the convergence in distribution of the sequence $F_n$ towards the target random variable $F_\infty$ in addition to convergence $\Delta(F_n) \to \Delta(F_\infty)=0$ one needs also the convergence of the third cumulant $\kappa_3(F_n) \to \kappa_3(F_\infty)=0$. Note that in this case we have $\dim_{\mathbb{Q}} \text{span} \{ \alpha^2_{\infty,1}, \alpha^2_{\infty,2} \} =1 < q=2$. 
}
\end{rem}

\begin{ex}\label{ex:thale}{ \rm
The aim of this simple example is to show that the requirement of separate convergences 
of the first $q+1$ cumulants is essential in Theorem \ref{thm:main-theorem1} as soon as $\dim_{\mathbb{Q}} \text{span} \{ \alpha^2_{\infty,1}, \cdots, \alpha^2_{\infty,q} \}  <  q$. Assume that $q=2$ and  $\alpha_{\infty,1}=- \alpha_{\infty,2}=1/2$. Consider the fixed sequence $$F_n=\alpha_{\infty,1} (N^2_1 -1) - \alpha_{\infty,2} (N^2_2 -1) \qquad n \ge1.$$ 
Then $\kappa_{2r}(F_n)=\kappa_{2r}(F_\infty)$ for all $r\ge1$,  in particular $\kappa_2(F_n)=\kappa_2(F_\infty)=1$, and $\Delta(F_n)=\Delta(F_\infty)=0$. However, it is easy to see that 
the sequence $F_n$ does not converges in distribution towards the target random variable $F_\infty$, because $2=\kappa_3(F_n)  \nrightarrow  \kappa_3(F_\infty)=0$. Note that in this example, we have $\dim_{\mathbb{Q}} \text{span} \{ \alpha^2_{\infty,1}, \alpha^2_{\infty,2} \} =1 < q=2$. %5Therefore the requirement of separate convergences 
%of the first $q+1$ cumulants is essential in Theorem \ref{thm:main-theorem1} as soon as $\dim_{\mathbb{Q}} \text{span} \{ \alpha^2_{\infty,1}, \cdots, \alpha^2_{\infty,q} \}  <  q$.
} 
\end{ex}
%{\color{red}
% \begin{ex}
%   \label{ex:bai}
% {\rm 
% I did not have time to go in details in Taqqu's stuff, however, I had
% a quick look at it (Theorem 2.4 in Bai and Taqqu paper)  and I think
% (needs double check) all the cumulants (bigger equal than 3) converge
% to the corresponding cumulants of the target distribution at rate
% $O(\lambda_1 - \lambda_2 -1)$, and therefore, following their
% notations, we have that as  
% \begin{equation*}
%   (\lambda_1,\lambda_2) \to (-1/2,-1/2) \mbox{ that }
% d_{W_2} (Z_{\lambda_1,\lambda_2}, Y_{\rho} (1)) = O( \sqrt{
% \lambda_1 - \lambda_2 -1 } ). 
% \end{equation*}
% $[Y_{\rho}$ (1) is given in (12), page 6, of their paper.] If you
% want, you can include it as another short example to the end of the
% Application section: second Wiener  chaos.  
% There are incredibly huge number of limit theorem results in
% mathematical statistics related to the Hurst parameter estimation of
% the Rosenblatt distribution (in the second Wiener with an infinite
% expansion). Do you guys see any possibility to somehow generalize our
% results to such target distributions ? In the case, it will have a
% substantial impact on Math-Stats community .  }
% 
%  }

\subsection{A lower bound on the $2$-Wasserstein distance in the second Wiener chaos}
\label{sec--LowBound}
\noindent
In this subsection, we detail how to upper bound the quantity $\Delta(F_n)$ with the 2-Wasserstein distance between $F_\infty$ and $F_n$ when $F_n$ and $F_\infty$ belong to the second Wiener chaos. First of all we recall some notations. The random variables $F_n$ and $F_\infty$ are defined by:
\begin{align}
F_n=\frac{1}{\sqrt{2}}\sum_{k\geq 1}\alpha_{n,k}(Z_k^2-1),\ F_\infty=\frac{1}{\sqrt{2}}\sum_{k=1}^q\alpha_{\infty,k}(Z_k^2-1),
\end{align}
where $(Z_k)$ is a sequence of iid standard normal random variables, $\{\alpha_{\infty,k}\}$ a collection of non-zero real numbers such that:
\begin{equation}\label{InfUS}
\sum_{k=1}^q\alpha_{\infty,k}^2=1.
\end{equation}
Similarly, we have:
\begin{equation}\label{SeqUS}
\sum_{k\geq 1}\alpha^2_{n,k}=1.
\end{equation}
From the previous assumptions, it is clear that $\kappa_2(F_n)=\kappa_2(F_\infty)=1$. It is also standard that the characteristic functions of $F_n$ and $F_\infty$ are analytic in the strips of the complex plane defined respectively by $D_n:=\{z\in\mathbb{C}: |\operatorname{Im}(z)|<1/(2\max |\alpha_{n,k}|)\}$ and $D_\infty:=\{z\in\mathbb{C}: |\operatorname{Im}(z)|<1/(2\max |\alpha_{\infty,k}|)\}$. In particular, by (\ref{InfUS}) and (\ref{SeqUS}), the characteristic functions of $F_n$ and $F_\infty$ are analytic in the strip $\{z\in\mathbb{C}: |\operatorname{Im}(z)|<1/2\}$. Moreover, in this strip of regularity, they admit the following integral representations:
\begin{align*}
\phi_{n}(z)&:=\int_{\mathbb{R}} e^{izx}\mu_{n}(dx),\\
\phi_{\infty}(z)&:=\int_{\mathbb{R}} e^{izx}\mu_{\infty}(dx).
\end{align*}
where $\mu_{n}$ and $\mu_{\infty}$ are the probability laws of $F_n$ and $F_\infty$ respectively. First, we give two technical lemmas.

\begin{lem}\label{tech1}
For any $x,y \in\mathbb{R}$ and $z\in\mathbb{C}$ such that $|z|=\rho$:
\begin{align}
|e^{izx}-e^{izy}|\leq \rho |x-y|e^{\rho (|x|+|y|)}.
\end{align}
\end{lem}

\begin{proof}
The proof is standard.
\end{proof}

\begin{lem}\label{tech2}
Let $X$ be a random variable belonging to the second Wiener chaos with unit variance. Then, we have:
\begin{align}
\mathbb{P}(|X|>x)\leq \exp(-x/e),
\end{align}
for all $x>e$.
\end{lem}

\begin{proof}
Since $X$ is in the second Wiener chaos, we have by hypercontractivity, for any $q>2$:
\begin{align}
\mathbb{E}[|X|^q]^{\frac{1}{q}}\leq (q-1)
\end{align}
Then, by Markov inequality, we have, for $x>e$:
\begin{align}
\mathbb{P}(|X|\geq x)\leq \dfrac{1}{x^q}\mathbb{E}[|X|^q]\leq \dfrac{1}{x^q}(q-1)^q
\end{align}
We choose $q=1+x/e$ and we obtain:
\begin{align}
\mathbb{P}(|X|\geq x)\leq e^{-x/e}.
\end{align}
\end{proof}
\noindent
We are now ready to the state the proposition linking the pointwise difference of the characteristic functions and of their derivatives with the 2-Wasserstein distance of $F_n$ and $F_\infty$.

\begin{prop}\label{Point2W}
  For any $\rho\in (0,1/(4e))$, there exists a strictly positive
  constant $C_{1,\rho}$ such that, for all $n\geq 1$ and for all
  $z\in\mathbb{C}$ with $|z|=\rho$, we have:
\begin{align}
|\phi_{n}(z)-\phi_{\infty}(z)|+|\phi'_{n}(z)-\phi'_{\infty}(z)|\leq \rho C_{1,\rho} W_2(F_n,F_\infty).
\end{align}
%Moreover, we have the following upper bound on $C_\rho$:
%\begin{align}
%C_\rho\leq 
%\end{align}
\end{prop}

\begin{proof}
By optimal transportation on the real line (Brenier Theorem), there exists a map $T_n$ such that we have:
\begin{align*}
W_2(F_n,F_\infty):=\bigg(\int_{\mathbb{R}}|x-T_n(x)|^2d\mu_\infty(x)\bigg)^{\frac{1}{2}}
\end{align*} 
Moreover, the push forward measure $\mu_\infty\circ T_n^{-1}$ is equal to $\mu_n$ so that, we have also:
\begin{align}
\phi_{n}(z):=\int_{\mathbb{R}} e^{izT_n(x)}\mu_{\infty}(dx),
\end{align}
for $z$ such that $|\operatorname{Im}(z)|<1/2$. Let $\rho\in (0,1/(4e))$ and $z\in\mathbb{C}$ such that $|z|=\rho$. We have:
\begin{align}
|\phi_{n}(z)-\phi_{\infty}(z)|\leq \int_{\mathbb{R}}|e^{izx}-e^{izT_n(x)}|\mu_{\infty}(dx).
\end{align}
Moreover, by Lemma \ref{tech1}, we have the following upper bound:
\begin{align}
|\phi_{n}(z)-\phi_{\infty}(z)|\leq \rho \int_{\mathbb{R}}|x-T_n(x)|e^{\rho (|x|+|T_n(x)|)}\mu_{\infty}(dx).
\end{align}
Using Cauchy-Schwarz inequality, we obtain:
\begin{align}
|\phi_{n}(z)-\phi_{\infty}(z)|&\leq \rho \bigg(\int_{\mathbb{R}}|x-T_n(x)|^2\mu_{\infty}(dx)\bigg)^{\frac{1}{2}}\bigg(\int_{\mathbb{R}}e^{2\rho (|x|+|T_n(x)|)}\mu_{\infty}(dx)\bigg)^{\frac{1}{2}},\\
&\leq  \rho W_2(F_n,F_\infty)\bigg(\int_{\mathbb{R}}e^{2\rho (|x|+|T_n(x)|)}\mu_{\infty}(dx)\bigg)^{\frac{1}{2}}.
\end{align}
Next, we need to prove that:
\begin{align}\label{unif}
\underset{n\geq 1}{\sup}\bigg(\int_{\mathbb{R}}e^{2\rho (|x|+|T_n(x)|)}\mu_{\infty}(dx)\bigg)<\infty.
\end{align}
By Lemma \ref{tech2}, we have that:
\begin{align}
\underset{n\geq 1}{\sup}(\mathbb{E}[e^{c|F_n|}])<\infty,
\end{align}
as soon as $c<1/e$. Since $\rho\in (0,1/(4e))$, (\ref{unif}) follows. To conclude the proof of the proposition, we need to bound similarly the pointwise difference of the derivatives of the characteristic functions. Since $F_n$ and $F_\infty$ are centered, we have:
\begin{align}
|\phi'_{n}(z)-\phi'_{\infty}(z)|\leq \int_{\mathbb{R}}|T_n(x)(e^{iz T_n(x)}-1)-x(e^{izx}-1)|d\mu_\infty(x).
\end{align}
Then, we have:
\begin{align}
|\phi'_{n}(z)-\phi'_{\infty}(z)|\leq (I)+(II),
\end{align}
with:
\begin{align}
(I)&:=\int_{\mathbb{R}}|x||e^{izx}-e^{izT_n(x)}|d\mu_\infty(x),\\
(II)&:=\int_{\mathbb{R}}|x-T_n(x)||e^{izT_n(x)}-1|d\mu_\infty(x).
\end{align}
For the first term, using Lemma \ref{tech1} and Cauchy-Schwarz inequality, we have the following bound:
\begin{align}
(I)\leq \rho W_2(F_n,F_\infty) \bigg(\int_{\mathbb{R}}|x|^2e^{2\rho (|x|+|T_n(x)|)}d\mu_\infty(x)\bigg)^{\frac{1}{2}}
\end{align}
Moreover, as previously, we have:
\begin{align*}
\underset{n\geq 1}{\sup}\bigg(\int_{\mathbb{R}}|x|^2e^{2\rho (|x|+|T_n(x)|)}d\mu_\infty(x)\bigg)<\infty,
\end{align*}
for $\rho\in (0,1/(4e))$. For the second term, we have:
\begin{align}
(II)&\leq \int_{\mathbb{R}}|x-T_n(x)|\rho |T_n(x)|e^{\rho |T_n(x)|}d\mu_\infty(x),\\
&\leq \rho W_2(F_n,F_\infty)\bigg(\int_{\mathbb{R}}|T_n(x)|^2e^{2\rho |T_n(x)|}d\mu_\infty(x)\bigg)^{\frac{1}{2}}.
\end{align}
Finally, we note that for $\rho\in (0,1/(4e))$:
\begin{align}
\underset{n\geq 1}{\sup}\bigg(\int_{\mathbb{R}}|T_n(x)|^2e^{2\rho |T_n(x)|}d\mu_\infty(x)\bigg)<\infty.
\end{align}
Taking
\begin{align}
C_{1,\rho}&:=\underset{n\geq 1}{\sup}\bigg(\int_{\mathbb{R}}|T_n(x)|^2e^{2\rho |T_n(x)|}d\mu_\infty(x)\bigg)^{\frac{1}{2}}\\
&+\underset{n\geq 1}{\sup}\bigg(\int_{\mathbb{R}}|x|^2e^{2\rho (|x|+|T_n(x)|)}d\mu_\infty(x)\bigg)^{\frac{1}{2}}\\
&+\underset{n\geq 1}{\sup}\bigg(\int_{\mathbb{R}}e^{2\rho (|x|+|T_n(x)|)}\mu_{\infty}(dx)\bigg)^{\frac{1}{2}}.
\end{align}
We obtain:
\begin{align}
|\phi_{n}(z)-\phi_{\infty}(z)|+|\phi'_{n}(z)-\phi'_{\infty}(z)|\leq \rho C_{1,\rho} W_2(F_n,F_\infty).
\end{align}
\end{proof}
\noindent
In order to upper bound the quantity $\Delta(F_n)$ with the 2-Wasserstein distance, we are going to use complex analysis together with Proposition \ref{Point2W}. First of all, recall the following inequality for the cumulants of $F_n$ and $F_\infty$:
\begin{align}
\forall r\geq 2,\ \mid\kappa_r(F_n)\mid &\leq 2^{r-1}(r-1)!\sum_{j=1}^{+\infty}\mid\alpha_{n,j}\mid^r,\\
&\leq 2^{r-1}(r-1)! \max \mid\alpha_{n,j}\mid^{r-2}\sum_{j=1}^{+\infty}\alpha_{n,j}^2,\\
&\leq 2^{r-1} (r-1)!
\end{align}
and similarly for $\kappa_r(F_\infty)$. Therefore the following series are convergent as soon as $|z|<1/2$:
\begin{align}
\sum_{r=2}^\infty \dfrac{\kappa_r(F_n)}{r!}(iz)^r,\ \sum_{r=2}^\infty \dfrac{\kappa_r(F_\infty)}{r!}(iz)^r
\end{align}
We are now ready to link the quantity $\Delta(F_n)$ with a certain functional on the difference of the characteristic functions.

\begin{prop}\label{LinkDn}
Let $\rho\in (0,1/2)$. There exists a strictly positive constant, $C_{2,\rho}>0$, such that:
\begin{align}
\int_0^{2\pi} |\dfrac{\phi_\infty'(\rho e^{i\theta})}{\phi_\infty(\rho e^{i\theta})}-\dfrac{\phi_n'(\rho e^{i\theta})}{\phi_n(\rho e^{i\theta})}|^2\frac{d\theta}{2\pi}\geq C_{2,\rho} \Delta(F_n)^2.
\end{align}
\end{prop} 

\begin{proof}
Let us fix $\rho\in (0,1/2)$. First of all, it is not difficult to see that we have the following identity as soon as $\mid z\mid<1/2$:
\begin{align}
\dfrac{\phi_{\infty}'(z)}{\phi_{\infty}(z)}-\dfrac{\phi_n'(z)}{\phi_n(z)}=\sum_{r=2}^\infty \dfrac{\kappa_r(F_\infty)-\kappa_r(F_n)}{(r-1)!}(i)^r z^{r-1}.
\end{align}
By orthogonality, we have the following identity:
\begin{align}
\int_0^{2\pi}\mid\dfrac{\phi_\infty'(\rho e^{i\theta})}{\phi_\infty(\rho e^{i\theta})}-\dfrac{\phi_n'(\rho e^{i\theta})}{\phi_n(\rho e^{i\theta})}\mid^2\frac{d\theta}{2\pi}=\sum_{r=2}^\infty\dfrac{\mid \kappa_r(F_\infty)-\kappa_r(F_n)\mid^2}{(r-1)!^2}\rho^{2(r-1)}.
\end{align}
Then, we obtain the following lower bound:
\begin{align}
\int_0^{2\pi}\mid\dfrac{\phi_\infty'(\rho e^{i\theta})}{\phi_\infty(\rho e^{i\theta})}-\dfrac{\phi_n'(\rho e^{i\theta})}{\phi_n(\rho e^{i\theta})}\mid^2\frac{d\theta}{2\pi}\geq C_{\rho}\sum_{r=2}^{2q+2}\frac{\mid \Theta_r\mid^2 }{2^{2(r-1)} (r-1)!^2}\mid \kappa_r(F_\infty)-\kappa_r(F_n)\mid^2.
\end{align}
for some $C_\rho>0$. This concludes the proof of the proposition.
\end{proof}
\noindent
We are now ready to state the main the result of this sub-section.
\begin{prop}\label{LowW2}
For any $\rho\in (0,1/(4e))$, there exists a strictly positive constant $C_{3,\rho}>0$ such that for all $n\geq 1$, we have:
\begin{align}
 W_2(F_n,F_\infty)\geq C_{3,\rho} \Delta(F_n).
\end{align}
\end{prop}

\begin{proof}
First of all, we note that for any $z\in\mathbb{C}$ such that $|z|=\rho$, we have:
\begin{align}
\mid\dfrac{\phi_\infty'(z)}{\phi_\infty(z)}-\dfrac{\phi_n'(z)}{\phi_n(z)}\mid \leq \frac{1}{|\phi_\infty(z)|}|\phi'_\infty(z)-\phi'_n(z)|+|\phi'_n(z)||\frac{1}{\phi_n(z)}-\frac{1}{\phi_\infty(z)}|
\end{align}
Moreover, it is clear that the function $\phi_\infty(z)$ is bounded away from $0$ on the disk centered at the origin and with radius $\rho$. Regarding the function $\phi_n(z)$, we have the following uniform bound (with $z$ on the disk centered at the origin and with radius $\rho$):
\begin{align}
|\frac{1}{\phi_n(z)}|&\leq \exp\bigg(\sum_{k=2}^{+\infty}\frac{1}{k!}|\kappa_k(F_n)| |z|^k\bigg),\\
&\leq \exp\bigg(\sum_{k=2}^{+\infty}\frac{2^{k-1}}{k}\rho^k\bigg):=\dfrac{e^{-\rho}}{\sqrt{1-2\rho}}.
\end{align}
Therefore, it is clear that:
\begin{align*}
\mid\dfrac{\phi_\infty'(z)}{\phi_\infty(z)}-\dfrac{\phi_n'(z)}{\phi_n(z)}\mid \leq C_{4,\rho}|\phi'_\infty(z)-\phi'_n(z)|+C_{5,\rho}|\phi_{\infty}(z)-\phi_n(z)|,
\end{align*}
for some strictly positive constants $C_{4,\rho}$ and  $C_{5,\rho}$ (independent of $n$). Thus, using Proposition \ref{Point2W}, we obtain:
\begin{align}
\mid\dfrac{\phi_\infty'(z)}{\phi_\infty(z)}-\dfrac{\phi_n'(z)}{\phi_n(z)}\mid \leq \rho C_{6,\rho} W_2(F_n,F_\infty).
\end{align}
Then, using Proposition \ref{LinkDn} concludes the proof of the proposition.
\end{proof}

\begin{rem}
\begin{itemize}
\item Combining Proposition \ref{LowW2} together with part (b) of Corollary \ref{cor:2wiener}, we obtain the fact that the convergence of $\Delta(F_n)$ to $0$ is equivalent to the convergence of $W_2(F_n,F_\infty)$ to $0$ when $\dim_{\mathbb{Q}} \text{span} \{\alpha^2_{\infty,1},\cdots,\alpha^2_{\infty,q} \} =q$. This complements the results contained in \cite{n-po-1,a-p-p} (see in particular Theorem $2$ of \cite{a-p-p}). Moreover, recall that convergence of $W_2(F_n,F_\infty)$ to $0$ is equivalent to convergence in distribution and convergence of the second moments. Therefore, when $\dim_{\mathbb{Q}} \text{span} \{\alpha^2_{\infty,1},\cdots,\alpha^2_{\infty,q} \} =q$ and $F_n$ and $F_\infty$ have unit variances, convergence in distribution of $F_n$ towards $F_\infty$ is equivalent to convergence of $\Delta(F_n)$ to $0$.
\item This justifies why we choose to study quantitative convergence result with respect to the 2-Wasserstein distance instead of other probability metrics such as Kolmogorov distance or 1-Wasserstein distance.
\end{itemize}
\end{rem}
\noindent
In the sequel, we provide a simple example for which it is possible to refine the previous lower bound. Let $(a_n)$ be a sequence of positive real numbers strictly less than $1$ which converges to $0$ when $n$ tends to infinity and such that:
\begin{align}
0<\overline{a}=\underset{n\geq 0}{\operatorname{sup}}(a_n)<1.
\end{align}
Then, we consider the following random variables:
\begin{align}
&F_n=\sqrt{\dfrac{1-a_n}{2}}(Z_1^2-1)+\sqrt{\dfrac{a_n}{2}}(Z^2_2-1),\\
&F_\infty=\dfrac{1}{\sqrt{2}}(Z^2-1).
\end{align}
We note that:
\begin{align}
\kappa_2(F_n)=\kappa_2(F_\infty)=1.
\end{align}
First of all, let us find an asymptotic equivalent for $\Delta(F_n,F_\infty)$. By definition, we have:
\begin{align}
\Delta(F_n,F_\infty)=\sum_{k=3}^4 \frac{\Theta_k}{\kappa_k(Z^2-1)}(\kappa_{k}(F_n)-\kappa_k(F_\infty)).
\end{align}
Since $\Theta_3=-\sqrt{2}$ and $\Theta_4=1$, we obtain:
\begin{align}
\Delta(F_n,F_\infty)=-\sqrt{2}\bigg[(\dfrac{1-a_n}{2})^{\frac{3}{2}}+(\dfrac{a_n}{2})^{\frac{3}{2}}-(\frac{1}{\sqrt{2}})^3\bigg]+\bigg[(\dfrac{1-a_n}{2})^{2}+(\dfrac{a_n}{2})^{2}-(\frac{1}{2})^2\bigg].
\end{align}
Then, one can prove that:
\begin{align}
\Delta(F_n,F_\infty)\sim \frac{a_n}{4}.
\end{align}
In order to find a comparable lower bound for the Wasserstein-2 distance, we need the following technical lemma.

\begin{lem}\label{D1-W2}
We denote by $\phi_n$ and $\phi_\infty$ the characteristic functions of $F_n$ and $F_\infty$ respectively. We have the following inequality:
\begin{align}\label{eq:D1-W2}
\underset{t\in \mathbb{R}\setminus \{0\}}{\sup}\dfrac{|\phi_n(t)-\phi_\infty(t)|}{|t|}\leq W_2(F_n,F_\infty).
\end{align}
\end{lem}

\begin{proof}
Let $T_n$ be as in the proof of Proposition \ref{Point2W} (given by Brenier theorem). We have:
\begin{align*}
|\phi_n(t)-\phi_\infty(t)|&:=|\int_{\mathbb{R}}e^{it T_n(x)}-e^{it x}d\mu_\infty(x)|,\\
&\leq \int_{\mathbb{R}}|e^{it (T_n(x)-x)}-1|d\mu_\infty(x),\\
&\leq |t|\int_{\mathbb{R}}|T_n(x)-x|d\mu_\infty(x),\\
&\leq |t| W_2(F_n,F_\infty),
\end{align*}
where we have used Cauchy-Schwarz inequality in the last inequality and the definition of $T_n$. This concludes the proof of the lemma.
\end{proof}
\noindent
Therefore, we have the following lower bound.

\begin{lem}\label{LowBouW2}
  There exists a strictly positive constant $c$ such that we have, for
  $n$ large enough:
\begin{align}
W_2(F_n,F_\infty)\geq c (a_n)^{\frac{3}{4}}.
\end{align}
\end{lem}

\begin{proof}
By straightforward computations, we have the following formula for  $\phi_n(t)$ and $\phi_\infty(t)$:
\begin{align}
\forall t\in \mathbb{R},\ \phi_n(t)&=\dfrac{e^{-it \sqrt{\dfrac{1-a_n}{2}}}}{\sqrt{1-2it \sqrt{\dfrac{1-a_n}{2}}}}\dfrac{e^{-it \sqrt{\dfrac{a_n}{2}}}}{\sqrt{1-2it \sqrt{\dfrac{a_n}{2}}}},\\
\phi_\infty(t)&=\dfrac{e^{-it \sqrt{\dfrac{1}{2}}}}{\sqrt{1-\sqrt{2}it}}.
\end{align}
Therefore, we have, for all $t\ne 0$:
\begin{align}
| \phi_n(t)-\phi_\infty(t)|&\geq ||\phi_n(t)|-|\phi_\infty(t)||:=|\dfrac{1}{((1+2t^2(1-a_n))(1+2t^2a_n))^{\frac{1}{4}}}-\dfrac{1}{(1+2t^2)^{\frac{1}{4}}}|,\\
&\geq |\dfrac{1}{(1+4t^4(1-a_n)a_n+2t^2)^{\frac{1}{4}}}-\dfrac{1}{(1+2t^2)^{\frac{1}{4}}}|,\\
&\geq \dfrac{1}{(1+2t^2)^{\frac{1}{4}}}|\dfrac{1}{(1+\frac{4t^4(1-a_n)a_n}{1+2t^2})^{\frac{1}{4}}}-1|
\end{align}
Now we select $t_n:=\frac{1}{\sqrt{a_n}}$. We obtain:
\begin{align}
| \phi_n(t_n)-\phi_\infty(t_n)|\geq \frac{(a_n)^{\frac{1}{4}}}{(2+a_n)^{\frac{1}{4}}}|\dfrac{1}{(1+\frac{4(1-a_n)}{a_n+2})^{\frac{1}{4}}}-1|.
\end{align}
The previous lower bound then implies:
\begin{align}
\dfrac{| \phi_n(t_n)-\phi_\infty(t_n)|}{|t_n|}\geq \frac{(a_n)^{\frac{3}{4}}}{(2+a_n)^{\frac{1}{4}}}|\dfrac{1}{(1+\frac{4(1-a_n)}{a_n+2})^{\frac{1}{4}}}-1|.
\end{align}
But it is clear that there exists a strictly positive constant $c>0$ (independent of $n$) such that:
\begin{align}
\frac{1}{(2+a_n)^{\frac{1}{4}}}|\dfrac{1}{(1+\frac{4(1-a_n)}{a_n+2})^{\frac{1}{4}}}-1|\geq c.
\end{align}
Then, we obtain that:
\begin{align}
\dfrac{| \phi_n(t_n)-\phi_\infty(t_n)|}{|t_n|}\geq c(a_n)^{\frac{3}{4}}.
\end{align}
Using Lemma \ref{D1-W2} concludes the proof of the lemma.
\end{proof}

\begin{rem}
Modifying the proof of Lemma \ref{LowBouW2} by choosing $t_n=(1/a_n)^\beta$ for some $\beta>0$ produces lower bounds with different rates of convergence to $0$. Indeed, one can check that the exponent of the resulting lower bound (denoted by $\chi(\beta)$) is defined in the following way:
\begin{align}
\chi(\beta)=\left\{
    \begin{array}{ll}
        1-\frac{\beta}{2} & \beta\in (0,\frac{1}{2}], \\
        \frac{3\beta}{2} & \beta\in [\frac{1}{2},+\infty).
    \end{array}
\right.
\end{align}
Thus, $\beta=\frac{1}{2}$ corresponds to the scale which reduces the most the gap between the lower and the upper scaling exponents.
\end{rem}

\subsection{Comparison  with the Malliavin--Stein method for the variance-Gamma}\label{subsec:MS-VG}
We recall that the target distributions of our interest laying in the
second Wiener chaos takes the form
\begin{equation}\label{target-wiener1}
  F_{\infty} = \sum_{i=1}^q \alpha_{\infty,i} (N^2_i -1),
\end{equation}
where $q \ge 2$, $\{N_i\}_{i=1}^{q}$ are i.i.d. $\mathscr{N}(0,1)$
random variables, and the coefficients $\{ \alpha_{\infty,i}\}_{i=1}^{q}$ are non-zero and distinct. We stress that $q$ in representation $(\ref{target-wiener1})$ cannot be infinity. The aim of this section is to study the connections between the class of our target distributions given as $(\ref{target-wiener1})$, and the so called variance-gamma class of probability distributions, and to compare our quantitive bound in Corollary \ref{cor:2wiener} with the bounds recently obtained in \cite{thale} using the Malliavin--Stein method. First, we recall some basic facts that we need on the variance-gamma probability distributions. For detailed information, we refer the reader to \cite{g-thesis,g-variance-gamma} and references therein. The random variable $X$ is said to have a variance-gamma probability distribution with parameters $r >0, \theta \in \R, \sigma >0, \mu \in \R$ if and only if its probability density function is given by 

\begin{equation*}
p_{\text{VG}} (x; r, \theta,\sigma,\mu) = \frac{1}{\sigma \sqrt{\pi} \Gamma(\frac{r}{2})} e^{ \frac{\theta}{\sigma^2} (x-\mu)} \left( \frac{\vert x -\mu \vert}{2 \sqrt{\theta^2 + \sigma^2}} \right)^{\frac{r-1}{2}} K_{\frac{r-1}{2}} \left( \frac{\sqrt{\theta^2 + \sigma^2}}{\sigma^2} \vert x -\mu \vert \right),
\end{equation*} 
where $x \in \R$, and $K_\nu(x)$ is a modified Bessel function of the second kind, and we write $X \sim \text{VG}(r,\theta,\sigma,\mu)$. Also, it is known that for $X \sim \text{VG}(r,\theta,\sigma,\mu)$ (see for example relation $(2.3)$ in \cite{g-variance-gamma})
\begin{equation}\label{eq:mean-variance}
\E(X)=\mu + r \theta, \quad \text{ and } \quad \text{Var}(X)=r (\sigma^2 + 2 \theta^2).
\end{equation}

\begin{lem}\label{lem:connection}
%What I write here in blue, leads to some contradictory thing that I cannot see where the mistake is. To summary, the message of the above red color argument is that as soon as $F_\infty$ is in the second Wiener chaos with $q=2$, $E(F^2_\infty) =1$, and assuming that $F_\infty$ belongs to the Variance-Gamma class with $\mu=0$, then necessary, first $\theta=0$, because of centering, and also $r=\sigma=1$, and finally $\alpha_{\infty,1}= - \alpha_{\infty,2} = \frac{1}{2}$. \\
(a) Let $N_1, N_2 \sim \mathscr{N}(0,1)$ be independent, and take two arbitrary $\alpha_{\infty,1}, \alpha_{\infty,2} > 0$. Then
\begin{equation}\label{eq:P}
F_\infty = \alpha_{\infty,1} (N^2_1 -1) - \alpha_{\infty,2} (N^2_2 -1) \sim \text{VG}(1, \alpha_{\infty,1} - \alpha_{\infty,2}, 2 \sqrt{\alpha_{\infty,1} \alpha_{\infty,2}}, \alpha_{\infty,2} - \alpha_{\infty,1}).
\end{equation}
%Note that, when $\alpha_{\infty,1} = \alpha_{\infty,2} > 0$, then in the Variance-Gamma representation $(\ref{eq:P})$ the parameters $\mu=\theta=0$. On the other hand, since $r=1$, and $\E(F^2_\infty)=1$, this implies that $\sigma = 2 \alpha_{\infty,1}=1$, i.e. $\alpha_{\infty,1} = 1/2$, and this reduces to the case $N_1 \times N_2$. 
%Hence, all the target random variables 
%$$F_\infty = \alpha_{\infty,1} (N^2_1 -1) - \alpha_{\infty,2} (N^2_2 -1)$$ with $\alpha_{\infty,1}, \alpha_{\infty,2}>0$ belong to the Variance-Gamma class.\\
(b) Let $q=2$. Then the target random variable $F_\infty$ as $(\ref{target-wiener1})$ so that $\alpha_{\infty,1}, \alpha_{\infty,2}>0$ (or similarly when $\alpha_{\infty,1}, \alpha_{\infty,2}<0$) cannot belong to the variance-gamma class.\\
(c) Let $q\ge 3$. Then the target random variable $F_\infty$ as $(\ref{target-wiener1})$ cannot belong to the variance-gamma class.\\
\end{lem}
\begin{proof}
 %A straightforward comparison between the characteristic function of $F_\infty$ and the one of the Variance-Gamma random variable (see, for example \cite[page $83$]{madan}) 

(a) Set 
$$X = \alpha_{\infty,1} N^2_1 \sim \Gamma(\frac{1}{2}, \frac{1}{2\alpha_{\infty,1}}), \quad \text{ and } \quad Y=\alpha_{\infty,2} N^2_2 \sim \Gamma(\frac{1}{2}, \frac{1}{2\alpha_{\infty,2}}).$$
Then the claim follows directly from part (v) in \cite[Proposition
3.8]{g-thesis}. (b,c) These also follow directly using a
straightforward comparison between the characteristic function of
$F_\infty$ and the one of the variance-gamma random variable (see, for
example, \cite[page $83$]{madan}).
\end{proof}
Next, we want to compare our bound in Corollary \ref{cor:2wiener} with
the bound in \cite{thale} obtained using the Malliavin--Stein
method. A good starting point for such comparison is the right hand
side of equation $(4.1)$ in \cite[Theorem 4.1]{thale}.  This is
because the bound in \cite[Corollary 5.10, part (a)]{thale} is
obtained from the right hand side of equation $(4.1)$ in \cite[Theorem
4.1]{thale} by norms of contraction operators. In virtue of Lemma
\ref{lem:connection}, in order for $F_\infty$ as in
$(\ref{target-wiener1})$ to belong to the variance-gamma class, it is
necessary to have $r=1$ and $q=2$. Letting $r=1$ in the right hand
side of equation $(4.1)$ in \cite[Theorem 4.1]{thale}, and taking into
account that $\kappa_2(F_\infty)=\sigma^2 + 2 \theta^2$, and
$\kappa_3(F_\infty)= 2\theta (3 \sigma^2 + 4 \theta^2)$, for an
element $F$ in the second Wiener chaos associated to the underlying
isonormal process $X$, we arrive at
\begin{equation*}
\begin{split}
{{\bf \rm W}_1}(F, F_\infty) & \le C_1  \E \Big\vert \Gamma_2(F) -2 \theta \Gamma_1(F) - \sigma^2 (F + \theta) \Big\vert  \\
& \qquad + C_2 \vert \kappa_2(F) - \kappa_2(F_\infty) \vert \\
& \le C_1 \E \Big\vert \sum_{r=1}^{3} \frac{P^{(r)}(0)}{r! 2^{r-1}} \big( \Gamma_{r-1}(F) - \E (\Gamma_{r-1}(F)) \big) \Big\vert \\
& \qquad + C_2 \vert \kappa_3(F) - 4 \theta \kappa_2(F) -2 \sigma^2 \theta \vert + C_3\vert \kappa_2(F) - \kappa_2(F_\infty) \vert \\
& \le C_1 \sqrt{\Delta(F)} + C_2 \sum_{r=2}^{3} \vert \kappa_r(F) - \kappa_{r}(F_\infty) \vert.
\end{split}
\end{equation*}
The last inequality is derived from the Cauchy-Schwarz inequality
together with \cite[Lemma 3.1]{a-p-p} where we used the fact that $F$
belongs to the second Wiener chaos.

%%%%%%%%%%%%%%%%%%%%%%%%%%%%   Section 3   %%%%%%%%%%%%%%%%%%%%%%%%%%%%%%%%%%%

\section{Applications}
\label{sec:applications}

\subsection{An example from $U$-statistics} 
{Under some degeneracy conditions, it is possible to observe the
  appearance of limiting distributions of the form $\sum_{k\ge 1}
  \alpha_{\infty,k} \big( N^2_k -1 \big)$ in the context of
  $U$-statistics. In this example, we restrict our attention to second order
  $U$-statistics. We refer the reader to 
   \cite[Chapter $5.5$ Section $5.5.2$]{Ser80} or to \cite[Chapter $11$ Corollary $11.5$]{Jan97}
  for full generality. Let $Z_i=I_1(h_i)$ be a sequence of i.i.d.\
  standard normal random variables supported by the isonormal Gaussian process $\rm X$, where $I_1$ is the Wiener-Itô integral of order $1$ and $\{h_i\}$ is an orthonormal basis of $\HH$. Let $a\ne 0$ be a real number. We
  consider the following second order $U$-statistic which has a
  degeneracy of order 1: 
\begin{align*}
U_n&=\dfrac{2a}{n(n-1)}\sum_{1\leq i<j\leq n}Z_iZ_j,\\
&=I_2\bigg(\dfrac{2a}{n(n-1)}\sum_{1\leq i<j\leq n}h_i\hat{\otimes} h_j\bigg).&
\end{align*}
A direct application of Theorem $5.5.2$ in \cite{Ser80} allows one to obtain:
\begin{align*}
nU_n(h)\Rightarrow a\big(Z_1^2-1\big).
\end{align*}
Using Corollary \ref{cor:2wiener}, we have the following result:
\begin{cor}\label{rate1}
For any $n\geq 3$, we have:
\begin{align*}
{{\bf \rm W}}_2\big(nU_n(h), a(Z_1^2-1)\big)\leq C\bigg( a^2\sqrt{\dfrac{n(n-3)}{(n-1)^3}}+\dfrac{2a^2}{n-1}\bigg).
\end{align*}
Namely, for $n$ large enough:
\begin{align*}
{{\bf \rm W}}_2\big(nU_n(h), a(Z_1^2-1)\big)= \mathcal{O}(\dfrac{1}{\sqrt{n}}).
\end{align*}
\end{cor}
\begin{proof}
By Corollary \ref{cor:2wiener}, we have:
\begin{align*}
{\bf \rm W}_2\big(nU_n(h), a(Z_1^2-1)\big)\leq C \bigg(\sqrt{\Delta\big(nU_n(h)\big)}+\mid \kappa_2(nU_n(h))-\kappa_2(a(Z_1^2-1))\mid\bigg).
\end{align*}
But,
\begin{align*}
\kappa_2(a(Z_1^2-1))&=2a^2,\\
\kappa_2(nU_n(h))&=\dfrac{2a^2n}{n-1},\\
\Delta\big(nU_n(h)\big) &=\sum_{r=2}^4 \dfrac{\Theta_r}{\kappa_r(Z_1^2-1)}\kappa_r\big(nU_n(h)\big),\\
\Delta\big(nU_n(h)\big) &=\sum_{r=2}^4 \dfrac{\Theta_r}{\kappa_r(Z_1^2-1)}\bigg[\kappa_r\big(nU_n(h)\big)-\kappa_r\big(a(Z_1^2-1)\big)\bigg].
\end{align*}
In order to obtain an explicit rate of convergence, we have to compute the cumulants of order $3$ and $4$ of the random variable $nU_n(h)$. Since $nU_n(h)$ is in the second order Wiener chaos, we can apply the following formula:
\begin{equation}\label{cumulants}
\kappa_r(I_2(f))=2^{r-1}(r-1)!\langle f\otimes_1...\otimes_1 f ; f\rangle,
\end{equation}
where there are $r-1$ copies of $f$ in $f\otimes_1...\otimes_1 f$. We note that:
\begin{align*}
\forall r\geq 2,\ \kappa_r\big(a(Z_1^2-1)\big)=a^r2^{r-1}(r-1)!.
\end{align*}
Let us compute the third and the fourth cumulants of $nU_n(h)$. By formula (\ref{cumulants}), we have:
\begin{align*}
\kappa_3\big(nU_n(h)\big)=\dfrac{2^6a^3}{(n-1)^3}\langle f_n\otimes_1 f_n;f_n\rangle,
\end{align*}
with,
\begin{align*}
f_n=\sum_{1\leq i<j\leq n}h_i\hat{\otimes} h_j.
\end{align*}
By standard computations, we have:
\begin{align*}
f_n\otimes_1 f_n= \frac{1}{16}\sum_{i\ne j}\sum_{k\ne l}\bigg(\delta_{ik}h_j\otimes h_l+\delta_{il}h_j\otimes h_k+\delta_{jk}h_i\otimes h_l+\delta_{jl}h_i\otimes h_k\bigg).
\end{align*}
We denote by $(I)$, $(II)$, $(III)$ and $(IV)$ the four associated double sums. The scalar product of $(I)$ with $f_n$ gives:
\begin{align*}
\langle (I);f_n\rangle &=\frac{1}{32}\sum_{i\ne j}\sum_{k\ne l}\sum_{m\ne o}\delta_{ik}\langle h_j\otimes h_l ;h_m \hat{\otimes} h_o\rangle,\\
&=\frac{1}{32}\sum_{i\ne j}\sum_{k\ne l}\sum_{m\ne o}\delta_{ik}\bigg(\frac{1}{2}\langle h_j\otimes h_l;h_m\otimes h_o\rangle+\frac{1}{2}\langle h_j\otimes h_l;h_o\otimes h_m\rangle\bigg),\\
&=\frac{1}{64}\sum_{i\ne j}\sum_{k\ne l}\sum_{m\ne o}\delta_{ik}\bigg(\delta_{jm}\delta_{lo}+\delta_{jo}\delta_{lm}\bigg),\\
&=\frac{1}{32}n(n-1)(n-2).
\end{align*}
The three other terms contribute in a similar way. Thus, we have:
\begin{equation}\label{Uk3}
\kappa_3\big(nU_n(h)\big)=\dfrac{2^3 a^3}{(n-1)^2}n(n-2).
\end{equation}
Similar computations for the fourth cumulants of $nU_n(h)$ lead to the following formula:
\begin{align*}
\kappa_4\big(nU_n(h)\big)=\dfrac{2^3 a^4 3!}{(n-1)^3}n(n-2)(n-3).
\end{align*}
Using the facts that $\Theta_2=a^2$, $\Theta_3=-2a$ and $\Theta_4=1$, we obtain:
\begin{align*}
\Delta\big(nU_n(h)\big) &=\sum_{r=2}^4 \dfrac{\Theta_r}{\kappa_r(Z_1^2-1)}\bigg[\kappa_r\big(nU_n(h)\big)-\kappa_r\big(a(Z_1^2-1)\big)\bigg],\\
&= \dfrac{a^4}{n-1}+\dfrac{2a^4}{(n-1)^2}+\dfrac{a^4}{(n-1)^3}[1+(3-2n)n],\\
&=\dfrac{a^4}{(n-1)^3}n(n-3).
\end{align*}
The result then follows.
\end{proof}
}
%\end{ex}

\subsection{Application to some quadratic forms}
{\rm In this example, we are interested in the asymptotic distributions of sequences of some specific quadratic forms. More precisely, we consider the following sequence of random variables:
\begin{align*}
Q_n(Z)=\sum_{i,j=1}^na_{i,j}(n)Z_iZ_j,
\end{align*}
where $A_n=\big(a_{i,j}(n)\big)$ is a $n\times n$ real-valued
symmetric matrix  and $(Z_i)$ an i.i.d.\ sequence of standard normal
random variables. A full description of the limiting distributions for
this type of sequences is contained in \cite{Sev61}. In particular, it
is possible to observe the appearance of limiting distributions of the
form $\sum_{k\ge 1} \alpha_{\infty,k} \big( N^2_k -1
\big)$. Sufficient conditions for such an appearance have been
introduced in \cite{WV73}. Let $\{\lambda_m,\ m\in \{1,...,q\}\}$ be
$q$ distinct non-zero real numbers. We make the following assumptions: 
\begin{itemize}
\item Let $\{b_i^m(n)\}$ be a sequence of real numbers such that:
\begin{align*}
&\sum_{i=1}^n b_i^m(n)b_i^k(n) \rightarrow \delta_{km},\\
&\exists b>0,\  \forall i,m,n,\ \sqrt{n}\mid b_i^m(n)\mid \leq b<+\infty
\end{align*}
\item For each $m$, we assume that:
\begin{align*}
\sum_{i,j=1}^n a_{i,j}(n)b_i^m(n)b_j^m(n) \rightarrow \lambda_m.
\end{align*}
\item Finally, we assume that:
\begin{align*}
\sum_{i,j}^na_{i,j}(n)^2\rightarrow \sum_{m=1}^q\lambda_m^2.
\end{align*}
\end{itemize}
In order to fit the assumptions of Corollary \ref{cor:2wiener}, we renormalize the quadratic form $Q_n$. We denote by $\tilde{Q}_n$ the quadratic form associated with the matrix $\tilde{A}_n$ defined by:
\begin{align*}
\tilde{a}_{i,j}(n)=\dfrac{a_{i,j}(n)}{\bigg(\sum_{i,j=1}^na_{i,j}(n)^2\bigg)^{\frac{1}{2}}}
\end{align*} 
\noindent
In particular, we have:
\begin{itemize}
\item for each $m\geq 1$,
\begin{align*}
\sum_{i,j=1}^n \tilde{a}_{i,j}(n)b_i^m(n)b_j^m(n) \rightarrow \tilde{\lambda}_m=\dfrac{\lambda_m}{\bigg(\sum_{m=1}^q\lambda^2_m\bigg)^{\frac{1}{2}}},
\end{align*}
\item and,
\begin{align*}
1=\sum_{i,j}^n\tilde{a}_{i,j}(n)^2\rightarrow \sum_{m=1}^q\tilde{\lambda}_m^2=1.
\end{align*}
\end{itemize}
By Theorem $2$ of \cite{WV73}, we have the following result:
\begin{align*}
\tilde{Q}_n(Z)-\mathbb{E}[\tilde{Q}_n(Z)] \Rightarrow \tilde{Q}_\infty=\sum_{m=1}^q\tilde{\lambda}_m(Z_m^2-1).
\end{align*}
If we assume that the $Z_i$ is a sequence of standard normal random variables supported by a Gaussian isonormal process, we have the following representation:
\begin{align*}
\tilde{Q}_n(Z)-\mathbb{E}[\tilde{Q}_n(Z)]=I_2\bigg(\sum_{i,j=1}^n\tilde{a}_{i,j}(n)h_i\hat{\otimes}h_j\bigg),
\end{align*}
with $Z_i=I_1(h_i)$. Applying Corollary \ref{cor:2wiener}, we will obtain an explicit rate of convergence for the previous limit theorem in $2$-Wasserstein distance. For this purpose we need to compute the cumulants of order $r$ of $\tilde{Q}_n(Z)-\mathbb{E}[\tilde{Q}_n(Z)]$ for $r\in \{2,...,2q+2\}$. Using the fact that the $Z_i$'s are i.i.d.\ standard normal, we have:
\begin{align*}
\forall r\in 2,...,2q+2,\ \kappa_r(\tilde{Q}_n(Z))=2^{r-1}(r-1)! \operatorname{Tr}\big(\tilde{A}^r_n\big).
\end{align*}
Combining the previous formula together with Corollary \ref{cor:2wiener}, we obtain the following bound on the 2-Wasserstein distance between $\tilde{Q}_n(Z)-\mathbb{E}[\tilde{Q}_n(Z)]$ and $\tilde{Q}_\infty$:
\begin{align}\label{GBound}
{\bf \rm W}_2\big(\tilde{Q}_n(Z)-\mathbb{E}[\tilde{Q}_n(Z)], \tilde{Q}_\infty\big)\leq &C\bigg(\sqrt{\sum_{r=2}^{2q+2}\Theta_r\bigg[\operatorname{Tr}\big(\tilde{A}^r_n\big)-\sum_{m=1}^q\tilde{\lambda}_m^r\bigg]}\\
&+\sum_{r=2}^{q+1}2^{r-1}(r-1)!\mid \operatorname{Tr}\big(\tilde{A}^r_n\big)-\sum_{m=1}^q\tilde{\lambda}_m^r \mid\bigg)\nonumber.
\end{align}
Thanks to this bound, we can obtain explicit rates of convergence for some more specific examples. In the sequel, we denote by $\mathcal{C}^{\alpha}\big([0,1]\big)$ the space of Hölder continuous real-valued functions of order $\alpha\in (0,1]$ on $[0,1]$. We have the following result.

\begin{cor}\label{Rate2}
Let $\{e_m\}$ be $q$ distinct orthonormal functions of $L^2(0,1)$ such that $e_m\in \mathcal{C}^{\alpha}\big([0,1]\big)$ for some $\alpha\in (0,1]$. Let $K_q(.,.)$ be the square integrable kernel defined by
\begin{align*}
\forall (x,y)\in (0,1)\times(0,1),\ K_q(x,y)=\sum_{m=1}^q\lambda_me_m(x)e_m(y)
\end{align*}
and let $A_n$ be the $n\times n$ matrix defined by:
\begin{align*}
\forall i,j,n,\ a_{i,j}(n)=\dfrac{1}{n}K_q(\frac{i}{n},\frac{j}{n}).
\end{align*}
Then, we have, for $n$ large enough:
\begin{align*}
{\bf \rm W}_2\big(\tilde{Q}_n(Z)-\mathbb{E}[\tilde{Q}_n(Z)], \tilde{Q}_\infty\big)=\mathcal{O}\big(\frac{1}{n^{\frac{\alpha}{2}}}\big).
\end{align*}
\end{cor}
\begin{proof}
First of all, choosing $b_i^m(n)=e_m(\frac{i}{n})/\sqrt{n}$, we note that the assumptions of the non-central limit theorem are verified so that the corresponding quadratic form converges in law towards $\sum_{m=1}^q\tilde{\lambda}_m\big(Z_m^2-1\big)$. Let us work out the bound (\ref{GBound}) in order to obtain an explicit rate of convergence. By standard computations, we have for all $r\geq 2$:
\begin{align*}
\operatorname{Tr}\big(A^r_n\big)&=\sum_{i_1,...,i_r}a_{i_1,i_2}(n)...a_{i_r,i_1}(n),\\
&=\sum_{m_1,...,m_r}\lambda_{m_1}...\lambda_{m_r}\frac{1}{n^r}\sum_{i_1,...,i_r}e_{m_1}\big(\frac{i_1}{n}\big)e_{m_1}\big(\frac{i_2}{n}\big)....e_{m_r}\big(\frac{i_r}{n}\big)e_{m_r}\big(\frac{i_1}{n}\big),\\
&=\sum_{m}\lambda_m^r\frac{1}{n^r}\sum_{i_1,...,i_r}e^2_m\big(\frac{i_1}{n}\big)...e^2_m\big(\frac{i_r}{n}\big)\\
&+\sum_{m_1,...,m_r}'\lambda_{m_1}...\lambda_{m_r}\frac{1}{n^r}\sum_{i_1,...,i_r}e_{m_1}\big(\frac{i_1}{n}\big)e_{m_1}\big(\frac{i_2}{n}\big)....e_{m_r}\big(\frac{i_r}{n}\big)e_{m_r}\big(\frac{i_1}{n}\big),
\end{align*}
where $\sum'$ means that we have excluded the hyper diagonal $\Delta_r=\{(m_1,...,m_r)\in \{1,...,q\}^r,\ m_1=...=m_r\}$. Thus, we have:
\begin{align*}
\operatorname{Tr}\big(A^r_n\big)-\sum_{m=1}^q\lambda_m^r&=\sum_{m}\lambda_m^r\bigg[\bigg(\frac{1}{n}\sum_{i=1}^{n}e^2_m\big(\frac{i}{n}\big)\bigg)^r-1\bigg]\\
&+\sum_{m_1,...,m_r}'\lambda_{m_1}...\lambda_{m_r}\frac{1}{n^r}\sum_{i_1,...,i_r}e_{m_1}\big(\frac{i_1}{n}\big)e_{m_1}\big(\frac{i_2}{n}\big)....e_{m_r}\big(\frac{i_r}{n}\big)e_{m_r}\big(\frac{i_1}{n}\big).
\end{align*}
Note that the second term tends to $0$ as $n$ tends to $+\infty$ since we have excluded the hyper diagonal $\Delta_r$ and that:
\begin{align*}
\frac{1}{n^r}\sum_{i_1,...,i_r}e_{m_1}\big(\frac{i_1}{n}\big)e_{m_1}\big(\frac{i_2}{n}\big)....e_{m_r}\big(\frac{i_r}{n}\big)e_{m_r}\big(\frac{i_1}{n}\big)\rightarrow \delta_{m_1,m_2}...\delta_{m_r,m_1}.
\end{align*}
Since $e_m\in \mathcal{C}^{\alpha}\big([0,1]\big)$, we have the following asymptotic for every $m$:
\begin{align*}
\mid \bigg(\frac{1}{n}\sum_{i=1}^{n}e^2_m\big(\frac{i}{n}\big)\bigg)^r-1\mid=\mathcal{O}\big(\dfrac{1}{n^\alpha}\big).
\end{align*}
Similarly, using the fact that $e_m\in \mathcal{C}^{\alpha}\big([0,1]\big)$, it is straightforward to see that the second term is $\mathcal{O}(1/n^{\alpha})$. Now, we note that:
\begin{align*}
\mid\operatorname{Tr}\big(\tilde{A}^r_n\big)-\sum_{m=1}^q\tilde{\lambda}_m^r\mid=\mathcal{O}\big(\dfrac{1}{n^\alpha}\big),
\end{align*}
since for $r\geq 2$,
\begin{align*}
\operatorname{Tr}\big(\tilde{A}^r_n\big)-\sum_{m=1}^q\tilde{\lambda}_m^r=\dfrac{\operatorname{Tr}\big(A^r_n\big)-\sum_{m=1}^q\lambda_m^r}{\bigg(\sum_{i,j}a^2_{i,j}\bigg)^{\frac{r}{2}}}+\sum_{m=1}^q\lambda^r_m\bigg(-\dfrac{1}{\bigg(\sum_{m=1}^q\lambda_m^2\bigg)^{\frac{r}{2}}}+\dfrac{1}{\bigg(\sum_{i,j}a_{i,j}^2\bigg)^{\frac{r}{2}}}\bigg).
\end{align*}
The result then follows.
\end{proof}

\begin{rem}\label{UniRes}
Theorem $2$ of \cite{WV73} is actually more general than the particular instance we have displayed since it holds for quadratic forms defined by:
\begin{align*}
\tilde{Q}_n(X)=\sum_{i,j=1}^n\tilde{a}_{i,j}(n)X_iX_j,
\end{align*}
where $(X_i)$ is an i.i.d.\ sequence of centered random variables such that $\mathbb{E}[X_i^2]=1$ and $\mathbb{E}[X_i^4]<+\infty$. Furthermore, since the works of Rotar' \cite{R73}, it is known that $\tilde{Q}_n(X)$ exhibits the same asymptotic behavior than $\tilde{Q}_n(Z)$ and explicit rates of approximation have been obtained in Kolmogorov distance (see e.g. \cite{GT05} and more generally \cite{MOO10} Theorems $2.1$ and $2.2$).  
\end{rem}
}
\noindent
We end this subsection with a universality result as announced in the previous remark. We assume that $(X_i)$ is a i.i.d.\ sequence of centered random variables such that $\mathbb{E}[X_i^2]=1$ and $\mathbb{E}[X_i^4]<+\infty$. First of all, as a direct application of Theorem $2.1$ of \cite{MOO10}, we obtain an explicit bound of approximation between $\tilde{Q}_n(X)$ and $\tilde{Q}_n(Z)$ in Kolmogorov distance.

\begin{cor}\label{Approx}
Under the previous assumptions, there exists $C>0$ such that:
\begin{align*}
d_{\operatorname{Kol}}\big(\tilde{Q}_n(X),\tilde{Q}_n(Z)\big)\leq \dfrac{C}{n^{\frac{1}{16}}}.
\end{align*}
\end{cor}
\begin{proof}
Since the sequence $(X_i)$ is a i.i.d.\ sequence of centered random variables with unit variance and finite $4th$ moment we have in particular that $\mathbb{E}[\mid X_i\mid^3]= \mathbb{E}[\mid X_1\mid^3]=\beta<+\infty$. Moreover, we have:
\begin{align*}
\tilde{Q}_n(x_1,...,x_n)=\sum_{i,j=1}^n\tilde{a}_{i,j}(n)x_ix_j=\sum_{S\subset [n]}c_S\prod_{i\in S}x_i,
\end{align*}
with,
\begin{equation}
  c_S=\left\{
      \begin{aligned}
       &0 & \mid\operatorname{S}\mid \ne 2\\
        &\tilde{a}_{i,j}(n) & \operatorname{S}=\{i,j\}.\\
      \end{aligned}
    \right.
\end{equation}
It is clear that $\sum_{S\subset [n]}c_S^2=1$. Moreover, we have, for any $i\in\{1,...,n\}$:
\begin{align*}
\sum_{S\ni i}c_S^2&=\sum_{j=1}^n\tilde{a}_{i,j}(n)^2,\\
&=\sum_{j=1}^n\dfrac{a^2_{i,j}(n)}{\sum_{i,j}a_{i,j}(n)^2},\\
&=\frac{1}{\sum_{i,j}a_{i,j}(n)^2}\frac{1}{n^2}\sum_{j=1}^n\bigg(\sum_m\lambda_me_m(\frac{i}{n})e_m(\frac{j}{n})\bigg)^2,\\
&\leq \frac{1}{\sum_{i,j}a_{i,j}(n)^2}\frac{q}{n^2}\sum_{j=1}^n\sum_m\lambda^2_me^2_m(\frac{i}{n})e^2_m(\frac{j}{n}),\\
&\leq \frac{C_1}{\sum_{i,j}a_{i,j}(n)^2}\frac{q}{n}\sum_m\lambda^2_m\bigg(\dfrac{1}{n}\sum_{j=1}^n e^2_m(\frac{j}{n})\bigg),
\end{align*}
where we have used the fact that the $e_m$ are bounded on $[0,1]$. Now, we note that:
\begin{align*}
&\sum_{i,j}a_{i,j}(n)^2\longrightarrow \sum_m\lambda^2_m,\\
&\sum_m\lambda^2_m\bigg(\dfrac{1}{n}\sum_{j=1}^n e^2_m(\frac{j}{n})\bigg)\longrightarrow \sum_m\lambda^2_m\bigg(\int_0^1e^2_m(x)dx\bigg)=\sum_m\lambda^2_m.
\end{align*}
Thus, we have, for some $C_2>0$:
\begin{align*}
\sum_{S\ni i}c_S^2\leq \frac{C_2}{n}.
\end{align*}
Applying directly Theorem $2.1$ of \cite{MOO10}, we obtain:
\begin{align*}
d_{\operatorname{Kol}}\big(\tilde{Q}_n(X),\tilde{Q}_n(Z)\big)\leq \dfrac{C}{n^{\frac{1}{16}}}.
\end{align*}
\end{proof}
\noindent
In order to obtain a rate for the Kolmogorov distance, we combine the previous corollary with Corollary \ref{Rate2} and with the fact that the Kolmogorov distance admits the following bound when the density of the target law is bounded (see e.g. Theorem $3.3$ of \cite{Chen-book}):
\begin{align*}
d_{\operatorname{Kol}}\big(\tilde{Q}_n(Z)-\mathbb{E}[\tilde{Q}_n(Z)],\tilde{Q}_\infty\big)\leq C \sqrt{{\bf \rm W}_2\big(\tilde{Q}_n(Z)-\mathbb{E}[\tilde{Q}_n(Z)], \tilde{Q}_\infty\big)}.
\end{align*}
$\tilde{Q}_\infty$ admits a bounded density as soon as $q$ is large enough. In this regard, we have the following lemma.

\begin{lem}\label{BoundedDensity}
Let $q\geq 3$. Let $X$ be a random variable such that:
\begin{align}
X=\sum_{j=1}^q \lambda_j (Z_j^2-1).
\end{align}
with $\{\lambda_j\}$ non-zero real numbers. Then, $X$ has a bounded density.
\end{lem}

\begin{proof}
The proof is standard so that we only sketch it. The characteristic function of $X$ is given by the following formula:
\begin{align}
\forall\xi\in\mathbb{R},\ \phi_X(\xi)=\prod_{j=1}^q \dfrac{\exp(-i\xi \lambda_j)}{\sqrt{1-2i\xi\lambda_j}}.
\end{align}
We introduce $\lambda_{min}=\min |\lambda_j|\ne 0$. Then,
\begin{align}
|\phi_X(\xi)| \leq \dfrac{1}{(1+4\xi^2\lambda_{min}^2)^{\frac{q}{4}}}.
\end{align}
Since $q\geq 3$, we deduce from the previous inequality that $\phi_X(.)$ is in $L^1(\mathbb{R})$. Thus, we can apply Fourier inversion formula to obtain the following bound:
\begin{align}
\|f_X\|_\infty\leq \|\phi_X\|_{L^1(\mathbb{R})}<\infty.
\end{align}
This concludes the proof of the lemma.
\end{proof}
\noindent
Therefore, we have the following result:

\begin{thm}\label{UNI}
Under the previous assumptions, we have:
\begin{align*}
d_{\operatorname{Kol}}\big(\tilde{Q}_n(X)-\tilde{E}[\tilde{Q}_n(X)],\tilde{Q}_\infty\big)\leq \dfrac{C_1}{n^{\frac{1}{16}}}+\dfrac{C_2}{n^{\frac{\alpha}{4}}}.
\end{align*}
\end{thm}

\begin{rem}\label{GenKolm}
We would like to mention that it is possible to combine the inequality (\ref{GBound}) together with Theorem $2.1$ of \cite{MOO10} to obtain a general bound in Kolmogorov distance for $q\geq 3$. We introduce the following quantity:
\begin{align*}
\tau_n=\underset{i\in \{1,...,n\}}{\sup}\sum_{j=1}^n\tilde{a}_{i,j}(n)^2
\end{align*}
Then,
\begin{align*}
d_{\operatorname{Kol}}\big(\tilde{Q}_n(X)-\tilde{E}[\tilde{Q}_n(X)],\tilde{Q}_\infty\big)&\leq C_1 (\tau_n)^{\frac{1}{16}}+C_2\bigg(\sqrt{\sum_{r=2}^{2q+2}\Theta_r\bigg[\operatorname{Tr}\big(\tilde{A}^r_n\big)-\sum_{m=1}^q\tilde{\lambda}_m^r\bigg]}\\
&+\sum_{r=2}^{q+1}2^{r-1}(r-1)!\mid \operatorname{Tr}\big(\tilde{A}^r_n\big)-\sum_{m=1}^q\tilde{\lambda}_m^r \mid\bigg)^{\frac{1}{2}}.
\end{align*}
\end{rem}
%\end{ex}

\subsection{The generalized Rosenblatt process at extreme critical exponent}
 We conclude this section with a more ambitious example, providing rates of
convergence in a recent result given by \cite[Theorem 2.4]{b-t}. Let $Z_{\gamma_1,\gamma_2}$ be the random variable defined by:
 \begin{align*}
  Z_{\gamma_1,\gamma_2}=\int_{\mathbb{R}^2}\bigg(\int_0^1(s-x_1)^{\gamma_1}_+(s-x_2)^{\gamma_2}_+ds\bigg)dB_{x_1}dB_{x_2},
 \end{align*}
 with $\gamma_i\in(-1,-1/2)$ and $\gamma_1+\gamma_2>-3/2$. By Proposition 3.1 of \cite{b-t}, we have the following formula for the cumulants of $Z_{\gamma_1,\gamma_2}$:
 \begin{equation*}
 \kappa_m\big(Z_{\gamma_1,\gamma_2}\big) =\frac{1}{2}(m-1)!A(\gamma_1,\gamma_2)^mC_m(\gamma_1,\gamma_2,1,1)
 \end{equation*}
  where, 
\begin{align*}
A(\gamma_1,\gamma_2)&=[(\gamma_1+\gamma_2+2)(2(\gamma_1+\gamma_2)+3)]^{\frac{1}{2}}\\
&\times[B(\gamma_1+1,-\gamma_1-\gamma_2-1)B(\gamma_2+1,-\gamma_1-\gamma_2-1)\\
&+B(\gamma_1+1,-2\gamma_1-1)B(\gamma_2+1,-2\gamma_2-1)]^{-\frac{1}{2}},
\end{align*}
and,
\begin{align*}
C_m(\gamma_1,\gamma_2,1,1)&=\sum_{\sigma\in\{1,2\}^m}\int_{(0,1)^m}\prod_{j=1}^m[(s_j-s_{j-1})^{\gamma_{\sigma_j}+\gamma_{\sigma'_{j-1}}+1}_+B(\gamma_{\sigma'_{j-1}}+1,-\gamma_{\sigma_j}-\gamma_{\sigma'_{j-1}}-1)\\
&+(s_{j-1}-s_j)^{\gamma_{\sigma_j}+\gamma_{\sigma'_{j-1}}+1}_+B(\gamma_{\sigma_j}+1,-\gamma_{\sigma_{j}}-\gamma_{\sigma'_{j-1}}-1)]ds_1...ds_m,\\
&B(\alpha,\beta)=\int_0^1u^{\alpha-1}(1-u)^{\beta-1}du.
\end{align*}
Let $\rho\in(0,1)$ and $Y_{\rho}$ be the random variable defined by:
\begin{align*}
Y_{\rho}=\dfrac{a_{\rho}}{\sqrt{2}}(Z_1^2-1)+\dfrac{b_{\rho}}{\sqrt{2}}(Z_2^2-1),
\end{align*}
with $Z_i$ independent standard normal random variables and $a_{\rho}$ and $b_{\rho}$ defined by:
\begin{align*}
&a_{\rho}=\dfrac{(\rho+1)^{-1}+(2\sqrt{\rho})^{-1}}{\sqrt{2(\rho+1)^{-2}+(2\rho)^{-1}}},\\
&b_{\rho}=\dfrac{(\rho+1)^{-1}-(2\sqrt{\rho})^{-1}}{\sqrt{2(\rho+1)^{-2}+(2\rho)^{-1}}}.
\end{align*}
For simplicity, we assume that $\gamma_1\geq \gamma_2$ and $\gamma_2=(\gamma_1+1/2)/\rho-1/2$. Then \cite[Theorem 2.4]{b-t} implies that as $\gamma_1$ tends to $-1/2$:
\begin{align}\label{eq:bai-taqqu}
Z_{\gamma_1,\gamma_2}\stackrel{\text{law}}{\to} Y_{\rho}.
\end{align}
 Note that, in this case, $\gamma_2$ automatically tends to $-1/2$ as well. To prove the previous result, the authors of \cite{b-t} prove the following convergence result:
\begin{align*}
\forall m\geq 2,\ \kappa_m\big(Z_{\gamma_1,\gamma_2}\big)\rightarrow\kappa_m\big(Y_{\rho}\big)=2^{\frac{m}{2}-1}(a_{\rho}^m+b_{\rho}^m)(m-1)!.
\end{align*} 
Now, using Corollary \ref{cor:2wiener}, Lemma \ref{lem:linear-combiniation} and applying Lemma \ref{lem:cumasym}, we can present the following quantative bound for convergence $(\ref{eq:bai-taqqu})$, namely as $\gamma_1$ 
tends to $-1/2$:
\begin{equation*}
\ud_{W_2} (Z_{\gamma_1,\gamma_2},Y_{\rho}) \le C_\rho \, \sqrt{- \gamma_1 - \frac{1}{2}},
\end{equation*}
where $C_{\rho}$ is some strictly positive constant depending on $\rho$ uniquely. In order to apply Corollary \ref{cor:2wiener} to obtain an explicit rate for convergence $(\ref{eq:bai-taqqu})$, we need to know at which speed $\kappa_m\big(Z_{\gamma_1,\gamma_2}\big)$ 
converges towards $\kappa_m\big(Y_{\rho}\big)$. For this purpose, we have the following lemma:
\begin{lem}\label{lem:cumasym}
Under the above assumptions, for any $m\geq 3$, we have, as $\gamma_1$ tends to $-1/2$:
\begin{align*}
\kappa_m\big(Z_{\gamma_1,\gamma_2}\big)=\kappa_m\big(Y_{\rho}\big)+O\big(-\gamma_1-\frac{1}{2}\big).
\end{align*}
\end{lem}
\begin{proof}
First of all, we note that, as $\gamma_1$ tends to $-1/2$:
\begin{align*}
A(\gamma_1,\gamma_2)&=[(\gamma_1+\frac{1}{\rho}(\gamma_1+\frac{1}{2})+\frac{3}{2})(2\gamma_1+\frac{2}{\rho}(\gamma_1+\frac{1}{2})+2)]^{\frac{1}{2}}\\
&\times[B(\gamma_1+1,-(1+\frac{1}{\rho})(\gamma_1+\frac{1}{2}))B(\frac{1}{\rho}(\gamma_1+\frac{1}{2})+\frac{1}{2},-(1+\frac{1}{\rho})(\gamma_1+\frac{1}{2}))\\
&+B(\gamma_1+1,-2\gamma_1-1)B(\frac{1}{\rho}(\gamma_1+\frac{1}{2})+\frac{1}{2},-\frac{2}{\rho}(\gamma_1+\frac{1}{2}))]^{-\frac{1}{2}},\\
&\approx \dfrac{(-\gamma_1-1/2)}{\sqrt{(1+\frac{1}{\rho})^{-2}+(\frac{4}{\rho})^{-1}}}-C_{\rho}(-3+2\gamma+2\psi\big(\frac{1}{2}\big))(\gamma_1+\frac{1}{2})^2\\
&+o((-\gamma_1-1/2)^2),
\end{align*}
where $\gamma$ is the Euler constant, $\psi(.)$ is the Digamma function and $C_{\rho}$ some strictly positive constant depending on $\rho$ uniquely. Note that $-3+2\gamma+2\psi\big(1/2\big)<0$. Moreover, we have:
\begin{align}
\nonumber C_m(\gamma_1,\gamma_2,1,1)&\approx \sum_{\sigma\in\{1,2\}^m}\int_{(0,1)^m}\prod_{j=1}^m\bigg\{\mathbb{I}_{s_j>s_{j-1}}\bigg[(-\gamma_{\sigma_j}-\gamma_{\sigma'_{j-1}}-1)^{-1}-\log(s_j-s_{j-1})\\
&\nonumber+(-\gamma-\psi\big(\frac{1}{2}\big))+o(1) \bigg]+\mathbb{I}_{s_j<s_{j-1}}\bigg[ (-\gamma_{\sigma_j}-\gamma_{\sigma'_{j-1}}-1)^{-1}-\log(s_{j-1}-s_j)\\
&+(-\gamma-\psi\big(\frac{1}{2}\big))+o(1)\bigg]\bigg\}ds_1...ds_m\\
\nonumber &\approx \sum_{\sigma\in\{1,2\}^m}\int_{(0,1)^m}\prod_{j=1}^m\bigg[(-\gamma_{\sigma_j}-\gamma_{\sigma'_{j-1}}-1)^{-1}+\mathbb{I}_{s_j>s_{j-1}}\log((s_j-s_{j-1})^{-1})\\
&+\mathbb{I}_{s_j<s_{j-1}}\log((s_{j-1}-s_j)^{-1})+(-\gamma-\psi\big(\frac{1}{2}\big))+o(1)\bigg]ds_1...ds_m.\label{approxCm}
\end{align}
Note that $-\gamma-\psi\big(\frac{1}{2}\big)>0$. The diverging terms in $C_m(\gamma_1,\gamma_2,1,1)$ are $B(\gamma_{\sigma'_{j-1}}+1,-\gamma_{\sigma_j}-\gamma_{\sigma'_{j-1}}-1)$ and $B(\gamma_{\sigma_j}+1,-\gamma_{\sigma_{j}}-\gamma_{\sigma'_{j-1}}-1)$. At $\sigma$ and $j$ fixed, the only possible values are:
\begin{align*}
&B(\gamma_1+1,-\gamma_{1}-\gamma_{2}-1)=B(\gamma_1+1,-(\gamma_1+\frac{1}{2})(1+\frac{1}{\rho})),\\
&\approx -\dfrac{1}{(1+\frac{1}{\rho})(\gamma_1+\frac{1}{2})}+(-\gamma-\psi(\frac{1}{2}))+o(1),\\
&B(\gamma_2+1,-\gamma_{1}-\gamma_{2}-1)=B(\frac{1}{\rho}(\gamma_1+\frac{1}{2})+\frac{1}{2},-(\gamma_1+\frac{1}{2})(1+\frac{1}{\rho})),\\
&\approx -\dfrac{1}{(1+\frac{1}{\rho})(\gamma_1+\frac{1}{2})}+(-\gamma-\psi(\frac{1}{2}))+o(1),\\
&B(\gamma_1+1,-2\gamma_{1}-1)\approx -\dfrac{1}{2(\gamma_1+\frac{1}{2})}+(-\gamma-\psi(\frac{1}{2}))+o(1),\\
&B(\gamma_2+1,-2\gamma_{2}-1)=B(\frac{1}{\rho}(\gamma_1+\frac{1}{2})+\frac{1}{2},-\frac{2}{\rho}(\gamma_1+\frac{1}{2})),\\
&\approx -\dfrac{\rho}{2(\gamma_1+\frac{1}{2})}+(-\gamma-\psi(\frac{1}{2}))+o(1).
\end{align*}
Moreover, we have, for $j$ fixed:
\begin{align*}
(s_j-s_{j-1})^{\gamma_{\sigma_j}+\gamma_{\sigma'_{j-1}}+1}_+&=\mathbb{I}_{s_j>s_{j-1}}(s_j-s_{j-1})^{\gamma_{\sigma_j}+\gamma_{\sigma'_{j-1}}+1}\\
&\approx \mathbb{I}_{s_j>s_{j-1}}[1+\log(s_j-s_{j-1})(\gamma_{\sigma_j}+\gamma_{\sigma'_{j-1}}+1)
+o((\gamma_{\sigma_j}+\gamma_{\sigma'_{j-1}}+1))].
\end{align*}
Developing the product in the right hand side of (\ref{approxCm}), we obtain: 
\begin{align*}
C_m(\gamma_1,\gamma_2,1,1)&\approx \sum_{\sigma\in\{1,2\}^m}\prod_{j=1}^m(-\gamma_{\sigma_j}-\gamma_{\sigma'_{j-1}}-1)^{-1}\\
&+(-\gamma-\psi\big(\frac{1}{2}\big))\sum_{\sigma\in\{1,2\}^m}\sum_{j=1}^m\prod_{k=1,\ k\ne j}^m(-\gamma_{\sigma_k}-\gamma_{\sigma'_{k-1}}-1)^{-1}\\
&+\sum_{\sigma\in\{1,2\}^m}\sum_{j=1}^m\prod_{k=1,\ k\ne j}^m(-\gamma_{\sigma_k}-\gamma_{\sigma'_{k-1}}-1)^{-1}\int_{(0,1)^m}\bigg[\mathbb{I}_{s_j>s_{j-1}}\log((s_j-s_{j-1})^{-1})\\
&+\mathbb{I}_{s_j<s_{j-1}}\log((s_{j-1}-s_j)^{-1})\bigg]ds_1...ds_m+o((-\gamma_1-\frac{1}{2})^{-m+1}),\\
&\approx \sum_{\sigma\in\{1,2\}^m}\prod_{j=1}^m(-\gamma_{\sigma_j}-\gamma_{\sigma'_{j-1}}-1)^{-1}\\
&+(-\gamma-\psi\big(\frac{1}{2}\big)+\frac{3}{2})\sum_{\sigma\in\{1,2\}^m}\sum_{j=1}^m\prod_{k=1,\ k\ne j}^m(-\gamma_{\sigma_k}-\gamma_{\sigma'_{k-1}}-1)^{-1}\\
&+o((-\gamma_1-\frac{1}{2})^{-m+1}).\\
\end{align*}
This leads to the following asymptotic for the cumulants of $Z_{\gamma_1,\gamma_2}$,
\begin{align*}
\kappa_m\big(Z_{\gamma_1,\gamma_2}\big)&\approx \frac{(m-1)!}{2}\bigg[\dfrac{(-\gamma_1-1/2)}{\sqrt{(1+\frac{1}{\rho})^{-2}+(\frac{4}{\rho})^{-1}}}-C_{\rho}(-3+2\gamma+2\psi\big(\frac{1}{2}\big))(\gamma_1+\frac{1}{2})^2\\
&+o((-\gamma_1-1/2)^2)\bigg]^m\bigg[\sum_{\sigma\in\{1,2\}^m}\prod_{j=1}^m(-\gamma_{\sigma_j}-\gamma_{\sigma'_{j-1}}-1)^{-1}\\
&+(-\gamma-\psi\big(\frac{1}{2}\big)+\frac{3}{2})\sum_{\sigma\in\{1,2\}^m}\sum_{j=1}^m\prod_{k=1,\ k\ne j}^m(-\gamma_{\sigma_k}-\gamma_{\sigma'_{k-1}}-1)^{-1}\\
&+o((-\gamma_1-\frac{1}{2})^{-m+1})\bigg],\\
&\approx \frac{(m-1)!}{2}\dfrac{(-\gamma_1-1/2)^m}{\bigg(\sqrt{(1+\frac{1}{\rho})^{-2}+(\frac{4}{\rho})^{-1}}\bigg)^m}\sum_{\sigma\in\{1,2\}^m}\prod_{j=1}^m(-\gamma_{\sigma_j}-\gamma_{\sigma'_{j-1}}-1)^{-1}\\
&+O((-\gamma_1-\frac{1}{2}))\\
&\approx \kappa_m\big(Y_{\rho}\big)+O((-\gamma_1-\frac{1}{2})),
\end{align*}
where we have used similar computations as in the proof of Theorem $2.4$ of \cite{b-t} for the last equality.
\end{proof}

 \section*{Acknowledgments}
 BA's research was supported by a Welcome Grant from the Universit\'e
 de Li\`ege.  YS gratefully acknowledges support by the Fonds de la
 Recherche Scientifique - FNRS under Grant MIS F.4539.16.

\addcontentsline{toc}{section}{References}%

\

(B. Arras) \textsc{Laboratoire Jacques-Louis Lions,  Universit\'e
    Pierre et Marie Curie, Paris, France}

(E. Azmoodeh) \textsc{Department of Mathematics and Statistics, University of
  Vaasa, Finland}

(G. Poly) \textsc{Institut de Recherche Math\'ematiques de Rennes,
  Universit\'e de Rennes 1, Rennes, France}
  
  (Y. Swan) \textsc{Mathematics department,  Universit\'e
    de Li\`ege, Li\`ege, Belgium}

\

\emph{E-mail address}, B. Arras {\tt arrasbenjamin@gmail.com}

\emph{E-mail address}, E. Azmoodeh {\tt ehsan.azmoodeh@uva.fi }

\emph{E-mail address}, G. Poly {\tt guillaume.poly@univ-rennes1.fr }

\emph{E-mail address}, Y. Swan  {\tt yswan@ulg.ac.be }


\begin{thebibliography}{99}

% \bibitem{a-c-p} 
% \textsc{Azmoodeh, A., Campese, S., and Poly, G.} (2014)
% \newblock Fourth Moment Theorems for Markov Diffusion Generators.
% \newblock \emph{J. Funct. Anal.} 266, No. 4, 2341-2359.

% \bibitem{a-m-m-p} \textsc{Azmoodeh, E., Malicet, D., Mijoule, G., Poly, G.} (2015)
% \newblock Generalization of the Nualart-Peccati criterion. \newblock To appear in \emph{Ann. Probab}.

\bibitem{a-p-p-2} \textsc{Arras, B., Azmoodeh, E., Poly, G. and Swan,
    Y. } (2017)
\newblock A Fourier approach to Stein characterizations
\newblock  \emph{in preparation}.

\bibitem{a-p-p} \textsc{Azmoodeh, E., Peccati, G., and Poly, G.} (2014)
\newblock Convergence towards linear combinations of chi-squared random variables: a Malliavin-based approach. 
\newblock  \emph{S\'eminaire de Probabilit\'es} XLVII (Special volume in
memory of Marc Yor), pp.339-367. 

% \bibitem{B90}
% \textsc{Barbour, A. D.} (1990)
% \newblock Stein's method for diffusion approximations.
% \newblock \emph{Probab. Theory Related Fields} 84, 297-322.

\bibitem{b-t}
\textsc{Bai, S., Taqqu, M.} (2015)
\newblock Behavior of the generalized Rosenblatt process at extreme critical exponent values.
\newblock {\em Ann. Probab.} 45(2), pp.1278--1324, 2017.

% \bibitem{B-H-J}
% \textsc{Barbour, A. D., Holst, L., and  Janson, S.} (1992) 
% \newblock \emph{Poisson approximation}. 
% \newblock Oxford: Clarendon Press. 

\bibitem{BU84}
\textsc{Borovkov, A.A. and Utev, S.A.} (1984)
\newblock  On an inequality and a related
characterization of the normal distribution.
\newblock \newblock {\em Theory Probab. Appl.}, 28(2), pp.219-228.

\bibitem{CP89} 
\textsc{Cacoullos, T. and Papathanasiou, V.} (1989).
 \newblock Characterizations of distributions by variance bounds.
 \newblock \emph{Statistics \& Probability Letters}, 7(5),
  pp.351-356.

% \bibitem{CM}
% \textsc{Chatterjee, S., and  Meckes, E.} (2008) 
% \newblock Multivariate normal approximation
% using exchangeable pairs. 
% \newblock \emph{ALEA} 4, 257-283.


  \bibitem{Chen-book}
  \textsc{Chen, L.HY,  Goldstein, L. and  Shao, Q.M.}  (2010)
  \newblock \emph{Normal approximation by Stein's method. }
  \newblock Springer Science  Business Media.

% \bibitem{dobler}
% \textsc{D\"obler, C.} (2015)
% \newblock Stein's method of exchangeable pairs for the Beta
% distribution and generalizations.
% \newblock \emph{Electron. J. Probab.} 20, No. 109, 1-34. 

% \bibitem{D-G-V}
% \textsc{D\"obler, C., Gaunt, R. E., and Vollmer, S. J.} (2015)
% \newblock An iterative technique for bounding derivatives of solutions
% of Stein equations. 
% \newblock Preprint \texttt{http://arxiv.org/abs/1510.02623}.

\bibitem{viquez}
\textsc{Eden, R.,  Viquez, J.} (2015)
\newblock Nourdin-Peccati analysis on Wiener and
Wiener-Poisson space for general distributions. 
\newblock  \emph{Stoch. Proc. Appl.} 125 (1), pp.182-216.


\bibitem{thale}
\textsc{Eichelsbacher, P., Th\"ale, C.} (2015)
\newblock Malliavin-Stein method for variance-gamma approximation on Wiener space.
\newblock \emph{Electron. J. Probab.} 20 (123), pp.1-28.

\bibitem{g-thesis}
\textsc{Gaunt, R. E.} (2013)
\newblock Rates of Convergence of variance-gamma approximations via Stein's method, 
\newblock Thesis at University of Oxford, the Queen's College.

\bibitem{g-2normal2}
\textsc{Gaunt, R. E.}  (2015)
\newblock Products of normal, beta and gamma random
variables: Stein characterisations and distributional theory. 
\newblock to appear in \emph{Braz. J. Probab. Stat.}

\bibitem{g-2normal}
\textsc{Gaunt, R. E.} (2017) 
\newblock On Stein's method for products of normal random variables and zero bias couplings. 
\newblock  {\em Bernoulli}, 23(4B), pp.3311-3345. 

\bibitem{g-variance-gamma}
\textsc{Gaunt, R. E.} (2014)
\newblock variance-gamma approximation via Stein's method.
\newblock \emph{Electron. J. Probab.}  19(38), pp.1-33.

\bibitem{GMS}
\textsc{Gaunt, R., Mijoule, G. and Swan, Y.} (2016)
\newblock  Stein operators for product
distributions, with applications. 
\newblock \emph{arXiv preprint arXiv:1604.06819.}

%\bibitem{g-p-r}
%\textsc{Gaunt, R. E., Pickett, A., and Reinert, G.} (2017)
%\newblock Chi-square approximation by Stein's method with application
%to Pearson's statistic. 
%\newblock To appear in \emph{Ann. Appl. Probab.}. 

 \bibitem{GT05}
 \textsc{G\"otze, F. and Tikhomirov, A. N.} (2005)
 \newblock Asymptotic expansions in non-central limit theorems for quadratic forms.
 \newblock \emph{J. Theoret. Probab.}  18(4), pp.757-811.
 
  %\bibitem{HLN14}
 %\textsc{Hu, Y. , Lu, F. and D. Nualart} (2014)
 %\newblock Convergence of densities of some functionals of Gaussian processes.
 %\newblock \emph{J. Func. Anal.}  266(2), pp.814-875.
 
\bibitem{Jan97}
\textsc{Janson S.} (1997)
\newblock \emph{ Gaussian Hilbert spaces}.  
\newblock Cambridge University Press: Volume 129. 

\bibitem{Kl85}
\textsc{Klaassen, C. A.} (1985)
\newblock  On an inequality of Chernoff. 
\newblock \emph{Ann. Probab.} 13(3), pp.966-974.

\bibitem{Kr17}
\textsc{Krein, C.} (2017). 
\newblock Weak convergence on Wiener space: targeting the
first two chaoses. 
\newblock \emph{arXiv preprint} arXiv:1701.06766.

\bibitem{Le12}
\textsc{Ledoux, M.} (2012). 
\newblock Chaos of a Markov operator and the fourth moment
condition. 
\newblock \emph{Ann. Probab.}, 40(6), pp.2439-2459.

 \bibitem{LRS}
 \textsc{Ley, C., Reinert, G., and Swan, Y.} (2017)
 \newblock Stein's method for comparison
of univariate distributions. 
 \newblock \emph{Probability Surveys} 14, pp.1-52.

% \bibitem{LS}
% \textsc{Ley, C.,   Swan, Y.} (2013) 
% \newblock Local Pinsker inequalities via
% Stein's discrete density approach.  
% \newblock \emph{IEEE Trans. Inf. Theory} 59, No. 9, 5584-5591.

 \bibitem{madan}
 \textsc{Madan, D. B., Carr, P. P.,Chang, E. C.} (1998)
 \newblock The variance gamma process and option pricing.
 \newblock \emph{Europ. Finance. Rev}, 2, pp.79-105.

% \bibitem{m-p-r-w}
% \textsc{Marinucci, D., Peccati, G.,  Rossi, M., Wigman, I.} (2015)
% \newblock Non-Universality of Nodal Length Distribution for Arithmetic Random Waves.
% \newblock Preprint \texttt{http://arxiv.org/abs/1508.00353}.

 \bibitem{MOO10}
 \textsc{Mossel E., O'Donnell R. and Oleszkiewicz K.} (2010)
 \newblock Noise stability of functions with low influences: invariance and optimality.
 \newblock \emph{Ann. of Math.} 171(1), pp.295-341.

\bibitem{n-pe-1} 
\textsc{Nourdin, I., Peccati, G.} (2012)
\newblock \emph{Normal Approximations Using Malliavin Calculus: from
  Stein's Method to Universality}.  
\newblock Cambridge Tracts in Mathematics. Cambridge University. 

%\bibitem{n-pe-2}
%\textsc{Nourdin, I., Peccati, G.} (2009)
%\newblock Noncentral convergence of multiple integrals. 
%\newblock \emph{Ann. Probab.} 37, No. 4, 1412-1426. 

\bibitem{n-pe-ptrf}
\textsc{Nourdin, I., Peccati, G.} (2009)
\newblock Stein's method on Wiener chaos.
\newblock \emph{ Probab. Theory Related Fields}. 145(1), pp.75-118.

\bibitem{n-pe-pams}
\textsc{Nourdin, I. and Peccati, G.} (2015)
\newblock The optimal fourth moment
theorem. 
\newblock \emph{Proceedings of the American Mathematical Society}, 143(7),
pp.3123-3133. 

%\bibitem{NPS}  
%\textsc{Nourdin, I., Peccati, G., and Swan, Y.} (2014).
%\newblock Entropy and the fourth
%moment phenomenon. 
%\newblock  \emph{J. Funct. Anal.}, 266, No. 5, 3170-3207.

\bibitem{n-po-1}
\textsc{Nourdin, I., Poly, G.} (2012)
\newblock Convergence in law in the second Wiener/Wigner chaos. 
\newblock \emph{Electron. Commun. Probab.} 17(36), pp.1-12.

\bibitem{nupe2005}
\textsc{Nualart, D., Peccati, G.} (2005). 
\newblock Central limit theorems for
sequences of multiple stochastic integrals. 
\newblock \emph{The Annals of Probability}, 33(1), pp.177-193.


\bibitem{p-r-r}
\textsc{Pek\"oz, E., R\"ollin, A. and Ross, N.} (2013)
\newblock  Degree asymptotics with rates for preferential attachment random graphs.
\newblock \emph{Ann. Appl. Prob.} 23, pp.1188-1218.

% \bibitem{pickett}
% \textsc{Pickett, A.} (2004) 
% \newblock \emph{Rates of Convergence of $\chi^2$ approximations via Stein's Method.}
% \newblock  DPhil thesis, University of Oxford.

% \bibitem{p-r-laplace}
% \textsc{Pike, J.,  Ren, H.} (2014)
% \newblock Stein's method and the Laplace distribution.
% \newblock Preprint \texttt{http://arxiv.org/abs/1210.5775}.

 \bibitem{R73}
 \textsc{Rotar', V. I. }(1973)
 \newblock Some limit theorems for polynomials of second degree.
 \newblock {\em Theory Probab. Appl.}, 18, pp.499-507.

\bibitem{Ser80}
\textsc{Serfling R. J.} (1980)
\newblock \emph{ Approximation theorems of mathematical statistics}.  
\newblock John Wiley \& Sons. 

\bibitem{Sev61}
\textsc{B.A. Sevast'yanov} (1961)
\newblock A class of limit distribution for quadratic forms of normal stochastic variables.
\newblock {\em Theory Probab. Appl.}, 6, pp.337--340.

% \bibitem{S72} 
% \textsc{Stein, C.} (1972) 
% \newblock A bound for the
%   error in the normal approximation to the distribution of a sum of
%   dependent random variables 
% \newblock Proceedings of the {S}ixth
%     {B}erkeley {S}ymposium on {M}athematical {S}tatistics and
%     {P}robability, {II}, {583--602}. 

\bibitem{S86}
\textsc{Stein, C.} (1986) 
\newblock \emph{Approximate computation of expectations.}
\newblock IMS, Lecture Notes-Monograph Series 7.

\bibitem{kusuoka}
\textsc{Kusuoka, S.,  Tudor, C. A.} (2012) 
\newblock Stein's method for invariant
measures of diffusions via Malliavin calculus. 
\newblock \emph{Stoch. Proc. Appl.} 122(4), pp.1627-1651.

\bibitem{WV73}
\textsc{Venter, J. H. and De Wet, T.}(1973)
\newblock Asymptotic distributions for quadratic forms with applications to tests of fit.
\newblock {\em Ann. Statist.}, 1(2), pp.380--387.

\bibitem{villani-book}
\textsc{Villani, C.} (2009)
\newblock \emph{Optimal transport. Old and new}. 
\newblock Grundlehren der Mathematischen Wissenschaften {338}. Springer.

\end{thebibliography}
\end{document}